\documentclass[final,1p,times]{elsarticle}

\usepackage{amsmath}
\usepackage{amssymb}
\usepackage{amsfonts}
\usepackage{amscd}
\usepackage{amsthm}
\usepackage{mathrsfs}

\numberwithin{equation}{section}

\allowdisplaybreaks[4]

\theoremstyle{plain}

\newtheorem{thm}{Theorem}[section]
\newtheorem{lem}{Lemma}[section]

\theoremstyle{definition}

\theoremstyle{remark}

\usepackage{tikz}
\usetikzlibrary{
	arrows.meta,
	positioning,
	shapes.multipart
}

\usepackage{graphicx}

\usepackage{pgfplots}
\pgfplotsset{compat=newest}

\usepackage{subcaption}

\usepackage{booktabs}
\usepackage{longtable}
\usepackage{adjustbox}

\usepackage{verbatim}
\usepackage{xcolor}
\usepackage{titletoc}
\usepackage{float}

\RequirePackage{CJKnumb}

\newcommand{\be}{\begin{equation}}
	\newcommand{\ee}{\end{equation}}

\biboptions{sort&compress}
\usepackage[
pagebackref,
colorlinks=true,
citecolor=green,
linkcolor=red,
urlcolor=blue
]{hyperref}

\journal{***}
\begin{document}
	\begin{frontmatter}
			\title{Stability and Hopf Bifurcation  of a Delayed SVIRS Epidemic Model with Media Coverage}
		\author{Songbo Hou \corref{cor1}}
	\ead{housb@cau.edu.cn}
	\address{Department of Applied Mathematics, College of Science, China Agricultural University,  Beijing, 100083, P.R. China}
	\author{Xinxin Tian}
	\ead{txx@cau.edu.cn}
	\address{Department of Applied Mathematics, College of Science, China Agricultural University,  Beijing, 100083, P.R. China}
	
	\cortext[cor1]{Corresponding author: Songbo Hou}
		\begin{abstract}
This paper formulates and analyzes a delayed SVIRS epidemic model incorporating media coverage effects, vaccination, waning immunity, temporary post-recovery immunity, saturated treatment, and delayed behavioral responses induced by media coverage. The positivity and uniform boundedness of solutions are established, the basic reproduction number is derived, and the local and global asymptotic stability of the disease-free and endemic equilibria is investigated. Taking the media-induced behavioral delay as the Hopf bifurcation parameter, a critical delay threshold is obtained, beyond which the endemic equilibrium loses stability and periodic oscillations emerge. Center manifold and normal form theories are applied to determine the direction of the local Hopf bifurcation and the stability of the bifurcating periodic solutions, while a global Hopf bifurcation theorem is used to establish the unbounded continuation of the periodic solution branch. Numerical simulations confirm the theoretical results and indicate that stronger media intervention can suppress epidemic oscillations and enhance system stability. These findings reveal the coupled effects of multiple epidemiological mechanisms and delayed media responses, providing theoretical support for the design of effective infectious disease control strategies.
	\end{abstract}	
	\end{frontmatter}
\section{Introduction}
Diseases such as SARS, H1N1, and COVID-19 have repeatedly caused major public health crises and economic disruptions worldwide over recent decades. Media coverage now plays an essential role in infectious disease prevention and control. Through various channels, including television, the internet, and social media platforms, epidemic-related information can spread rapidly and widely among the public. This not only encourages individuals to adopt protective measures—such as mask-wearing, vaccination, and social distancing—to reduce the risk of infection but also strongly influences the formulation and implementation of public health interventions and policies \cite{WANG20231,Xiao2015Media}. In this context, mathematical modelling provides an effective approach for quantitatively characterising media-related mechanisms and systematically investigating their effects on disease transmission. Such studies are of considerable theoretical and practical importance for improving epidemic control strategies and responding to public health emergencies \cite{mma.5438}.

The public's response to infectious diseases largely depends on their perception of risk. After obtaining epidemic-related information and understanding transmission mechanisms through media reports, individuals tend to adopt protective measures, such as reducing social contact and enhancing personal protection, to lower the risk of infection. Existing studies have shown that media coverage and health education can effectively suppress disease transmission. Enhanced media intervention may significantly reduce infection rates, while changes in individual behavior are closely associated with the amount of epidemic information received. Therefore, disseminating prevention and control measures through media channels has emerged as a key strategy for controlling the spread of infectious diseases \cite{Rahman2007Media,Cui2008SIS,Kiss2010Impact}.

To quantitatively characterize these processes, many infectious disease models incorporating media effects have been developed. The transmission coefficient is often modeled as a nonlinear decreasing function of the infected population. In \cite{Cui2008Impact}, Cui et al. modeled media influence using an exponential decay term of the form $\mu e^{-mI}$. Li and Cui adopted the saturation function $\beta_1-\dfrac{\beta_2 I}{m+I}$ to characterize the reduction in contact rate caused by media effects \cite{Li2009Effect}. Xiao et al. further considered the rate of change in case numbers and introduced a more complex media-related function $e^{-M(I,dI/dt)}$ \cite{Xiao2013Dynamics}. These studies reveal, from different perspectives, how media interventions regulate transmission dynamics and provide important insights into the effects of media coverage on epidemic spread. Additional related studies can be found in \cite{Liu2007Media,Collinson2014Modelling,Yan2016Media,eltit,Song2018}.

\begin{figure}[htbp]
	\centering
	\includegraphics[width=0.7\textwidth]{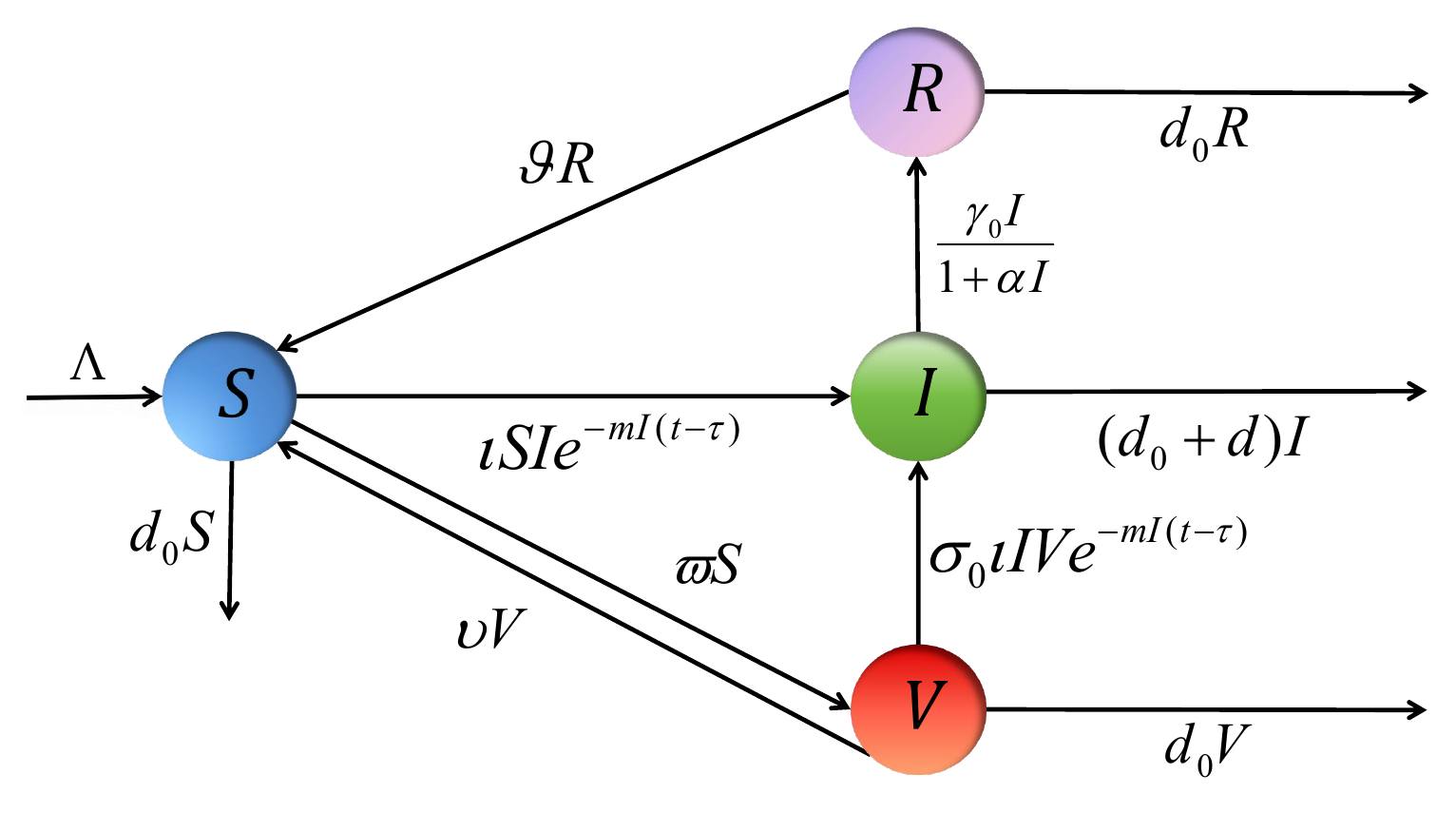}
	\caption{Compartmental structure of the proposed SVIRS model}
	\label{fig:00}
\end{figure}

However, in practice, both media reporting and public responses involve significant time delays. Such delays are reflected not only in the release of epidemic information but also in the process by which the public receives information and adjusts protective behaviors accordingly. Therefore, incorporating time delays into epidemic models can better reflect the actual disease transmission process. Recently, researchers have increasingly incorporated delays into media-related epidemic models to investigate the effects of delayed responses on system dynamics. For example, Song et al. introduced dual time delays into a media-influenced epidemic model and analyzed the resulting delay-induced local and global bifurcations \cite{Song2019}. Ma et al. showed in a behavioral intervention model that increasing delays may trigger Hopf bifurcation \cite{260195}. Zhao et al. demonstrated that, in an SIR model with media delay, local Hopf bifurcation may extend to global bifurcation once the delay exceeds a critical threshold \cite{Zhao2017}. Related studies can also be found in \cite{jie,WANG20231}.

Although considerable progress has been achieved in delayed epidemic models with media effects, many existing studies neglect waning vaccine-induced immunity, temporary immunity loss in recovered individuals, and treatment saturation caused by limited medical resources. As a consequence, the disease dynamics generated by complex prevention and control mechanisms may not be comprehensively represented within these modeling frameworks. At the same time, vaccination, as one of the core prevention and control measures, has been extensively studied in epidemiological models. In recent years, numerous scholars have conducted in-depth research on vaccination models and constructed various SVIR-type compartmental models. To study the prevention of pertussis and tuberculosis, Kribs-Zaleta et al. extended the SIS model by adding a vaccinated compartment $V$ \cite{Kribs-Zaleta2000}. Wang et al. established an age-structured SVIR model to discuss susceptibility characteristics and vaccine effectiveness \cite{10.1093/imamat/hxx020}. Additional studies on vaccine effects can be found in \cite{AZMI2026114250,article}. Notably, vaccine protection wanes over time, and the immunity acquired by recovered individuals is not permanent \cite{goel2020deterministic,doi:10.1142/S0218339025500093}. Additionally, when infection scales are large, medical resources can become saturated, leading to a decrease in the actual cure rate \cite{Kumar2019Nonlinear}.

Based on the above research, this paper comprehensively extends the classic SVIR model \cite{Xinjie} to construct an SVIRS infectious disease model  incorporating media-related delay, waning vaccine-induced immunity, waning temporary immunity in recovered individuals, and saturation of medical resources.  The model incorporates the exponential delay term $e^{-mI(t-\tau)}$ into the transmission process to capture the time-lag effects associated with media coverage and public reactions, while a saturated recovery rate is adopted to describe constrained medical resources. The proposed epidemic model is formulated as follows:
\begin{equation}\label{1.1}
	\left\{
	\begin{aligned}
		\frac{dS(t)}{dt} &= \Lambda - \iota I(t)e^{-mI(t-\tau)}S(t) - ( d_0 + \varpi)S(t) + \upsilon V(t) + \vartheta R(t), \\
		\frac{dV(t)}{dt} &= \varpi S(t) - \sigma_0  \iota I(t)e^{-mI(t-\tau)}V(t) - ( d_0 + \upsilon)V(t), \\
		\frac{dI(t)}{dt} &= \iota I(t)e^{-mI(t-\tau)}\big(S(t) + \sigma_0 V(t)\big) - ( d_0 + \gamma_0  + d)I(t), \\
		\frac{dR(t)}{dt} &= \frac{\gamma_0 I(t)}{1 + \alpha I(t)} - ( d_0 + \vartheta)R(t).
	\end{aligned}
	\right.
\end{equation}Here, $S(t)$, $V(t)$, $I(t)$, and $R(t)$ denote the numbers of susceptible, vaccinated, infected, and recovered individuals at time $t$, respectively. All parameters are positive, where $\Lambda $ denotes the population inflow rate, $\vartheta$ denotes the rate of immunity waning in recovered individuals, $ \iota$ represents the disease transmission rate, $\varpi$ is the vaccination rate, $m$ measures the media influence intensity (with larger $m$ indicating stronger media-induced reduction in infection rate), $\sigma_0\in[0,1]$ denotes the relative susceptibility of vaccinated individuals, $\upsilon$ denotes the rate of vaccine-induced immunity waning, $ d_0$ is the natural mortality rate, $\gamma_0 $ is the recovery rate, $\alpha$ represents the medical resource saturation coefficient, $d$ is the disease-induced death rate, and $\tau$ denotes the time delay in media influence.

Let $C=C([-\tau,0],\mathbb R^4)$ be the Banach space of continuous functions \(\psi=(\psi_1,\psi_2,\psi_3,\psi_4)\) equipped with the norm \[\|\psi\|=\sup_{-\tau\le \theta\le0}\max_{1\le j\le4}|\psi_j(\theta)|.\] Define \[C_+=\{\psi\in C:\psi_j(\theta)\ge0,\ -\tau\le \theta\le0,\ j=1,\dots,4\}.\] The system \eqref{1.1} is subject to the initial conditions: \begin{equation}\label{1.2}	S(\theta) = \psi_1(\theta),\quad V(\theta) = \psi_2(\theta),\quad I(\theta) = \psi_3(\theta),\quad R(\theta) = \psi_4(\theta),\quad -\tau \leq \theta \leq 0,\end{equation}
where $\psi \in C_+$.

The rest of this paper is structured as follows. Section 2 is devoted to the analysis of the well-posedness and boundedness of solutions, together with the investigation of the existence of the disease-free and endemic equilibria. The local stability analysis of these equilibria is presented in Section 3. Section 4 is devoted to the investigation of global stability. In Section 5, we study the occurrence of local Hopf bifurcations induced by the time delay. The global continuation of bifurcating periodic solutions is analyzed in Section 6. Numerical simulations are provided in Section 7 to support the theoretical results.  Finally, Section 8 provides concluding remarks.
\section{Preliminaries}
\subsection{Positivity}

\begin{thm}
	Suppose that $S(0)$, $V(0)$, $I(0)$ and $R(0)$  are all positive.  Then the solution generated by system \eqref{1.1} remains positive  for all \(t\ge0\).
\end{thm}
\begin{proof}
	Let $B(t)=\min \{S(t), V(t), I(t), R(t)\} $
	for all \(t \ge 0\).
	Since \(B(0) > 0\), it suffices to show that
	\(B(t) > 0\) for every \(t \ge 0\).
	
	Suppose, to the contrary, that there exists a time \(t_1^{*} > 0\) such that
	\(B(t) > 0\) for \(0 \le t < t_1^{*}\), whereas \(B(t_1^{*}) = 0\). At least one component vanishes at $t=t_1^*$. We consider the following possible cases.
	
	(1) If \( B({t_1}^*) = S({t_1}^*) \), then  
	\[
	\begin{aligned}
		\frac{dS(t)}{dt} &= \Lambda  - \iota I(t)e^{-mI(t-\tau)}S(t) - ( d_0 +\varpi)S(t) + \upsilon V(t) + \vartheta R(t) \\
		&\ge - \iota I(t)e^{-mI(t-\tau)}S(t) - ( d_0 + \varpi)S(t) \\
		&\ge -\max_{0\leq t\leq {t_1}^*}\left\{\iota I(t) e^{-mI(t-\tau)}\right\}S(t) - ( d_0 + \varpi)S(t) \\
		&= -b_1 S(t),
	\end{aligned}
	\]  
	for \( t \in [0, {t_1}^*] \), where \( b_1 = \max_{0\leq t\leq {t_1}^*}\left\{\iota I(t) e^{-mI(t-\tau)}\right\} + ( d_0 + \varpi) \). Therefore, \( S({t_1}^*) \ge S(0)e^{-b_1 {t_1}^*} > 0 \), which contradicts \( S({t_1}^*) = 0 \).  
	
	(2) If \( B({t_1}^*) = V({t_1}^*) \), then  
	\[
	\begin{aligned}
		\frac{dV(t)}{dt} &= \varpi S(t) - \sigma_0 \iota I(t)e^{-mI(t-\tau)}V(t) - ( d_0 + \upsilon)V(t) \\
		&\ge - \sigma_0 \iota I(t)e^{-mI(t-\tau)}V(t) - ( d_0 + \upsilon)V(t) \\
		&\ge -\max_{0\leq t\leq {t_1}^*}\left\{\sigma_0 \iota I(t)e^{-mI(t-\tau)}\right\}V(t) - ( d_0 + \upsilon)V(t) \\
		&= -b_2 V(t),
	\end{aligned}
	\]  
	for \( t \in [0, {t_1}^*] \), where \( b_2 = \max_{0\leq t\leq {t_1}^*}\left\{\sigma_0\iota I(t) e^{-mI(t-\tau)}\right\} + ( d_0 + \upsilon) \). Therefore, \( V({t_1}^*) \ge V(0)e^{-b_2 {t_1}^*} > 0 \), which contradicts \( V({t_1}^*) = 0 \).  
	
	(3) If \( B({t_1}^*) = I({t_1}^*) \), then  
	\[
	\begin{aligned}
		\frac{dI(t)}{dt} &= \iota I(t)e^{-mI(t-\tau)}\big(S(t) + \sigma_0 V(t)\big) - ( d_0 + \gamma_0  + d)I(t) \\
		&\ge - ( d_0 + \gamma_0  + d)I(t) \\
		&= -b_3 I(t),
	\end{aligned}
	\]  
	for \( t \in [0, {t_1}^*] \), where \( b_3 =  d_0 + \gamma_0  + d \). An application of the comparison theorem yields \( I({t_1}^*) \ge I(0)e^{-b_3 {t_1}^*} > 0 \), which contradicts \( I({t_1}^*) = 0 \).  
	
	(4) If \( B({t_1}^*) = R({t_1}^*) \), then  
	\[
	\begin{aligned}
		\frac{dR(t)}{dt} &= \frac{\gamma_0  I(t)}{1 + \alpha I(t)} - ( d_0 + \vartheta)R(t) \\
		&\ge - ( d_0 + \vartheta)R(t) \\
		&= -b_4 R(t),
	\end{aligned}
	\]  
	for \( t \in [0, {t_1}^*] \), where \( b_4 =  d_0 + \vartheta \). Therefore, \( R({t_1}^*) \ge R(0)e^{-b_4 {t_1}^*} > 0 \), which contradicts \( R({t_1}^*) = 0 \).  
	
	In summary, the proof is complete.
\end{proof}
\subsection{Boundedness}
\begin{thm}
	Suppose that $S(0)$, $V(0)$, $I(0)$ and $R(0)$  are all positive. Then every solution of system \eqref{1.1} remains bounded for all \(t\geq0\).
\end{thm}
\begin{proof}
	Let $M(t)=S(t)+V(t)+I(t)+R(t)$. Differentiating \(M(t)\) along the solutions of system \eqref{1.1} gives
	
	\[
	\begin{aligned}
		\frac{dM(t)}{dt} 
		&= \frac{dS(t)}{dt} + \frac{dV(t)}{dt} + \frac{dI(t)}{dt} + \frac{dR(t)}{dt} \\[4pt]
		&= \Lambda  - \iota I(t)e^{-mI(t-\tau)}S(t) - ( d_0 + \varpi)S(t) + \upsilon V(t) + \vartheta R(t)\\
		&\quad +\varpi S(t) - \sigma_0 \iota I(t)e^{-mI(t-\tau)}V(t) - ( d_0 + \upsilon)V(t)  \\
		&\quad +\iota I(t)e^{-mI(t-\tau)}\big(S(t) + \sigma_0 V(t)\big) - ( d_0 + \gamma_0  + d)I(t) \\
		&\quad + \frac{\gamma_0  I(t)}{1 + \alpha I(t)} - ( d_0 + \vartheta)R(t)\\[4pt]
		&= \Lambda  -  d_0 M(t) - d I(t) + \gamma_0 \left[ \frac{I(t)}{1+\alpha I(t)} - I(t) \right].
	\end{aligned}
	\]
	
	Observing that  
	
	\[
	\frac{\gamma_0  I}{1+\alpha I} - \gamma_0  I = -\frac{\gamma_0  \alpha I^2}{1+\alpha I} \le 0,
	\] 
	we obtain the differential inequality  
	
	\[
	\frac{dM(t)}{dt} \le \Lambda  -  d_0 M(t).
	\]
	
By the comparison principle, \[ M(t) \leq M(0)e^{-d_0t} +\frac{\Lambda}{d_0} \left(1-e^{-d_0t}\right), \qquad t\geq0. \] Consequently, \[ M(t) \leq \max\left\{ M(0),\frac{\Lambda}{d_0} \right\}, \qquad t\geq0, \] and \[ \limsup_{t\to\infty}M(t) \leq \frac{\Lambda}{d_0}. \] 

Moreover, the region \[ \Omega = \left\{ (S,V,I,R)\in\mathbb R_+^4: S+V+I+R\leq\frac{\Lambda}{d_0} \right\} \] is positively invariant. Indeed, if \(M(0)\leq\Lambda/d_0\), then \[ M(t)\leq M(0)e^{-d_0t} +\frac{\Lambda}{d_0}\left(1-e^{-d_0t}\right) \leq\frac{\Lambda}{d_0} \] for all \(t\geq0\).

Since \(S(t)\), \(V(t)\), \(I(t)\), and \(R(t)\) are nonnegative and each of them is bounded above by \(M(t)\), all components of the solution are bounded on \([0,\infty)\). This completes the proof.

\end{proof}
\subsection{Existence of Equilibrium Points}
The existence of an equilibrium for system \eqref{1.1} can be demonstrated by applying the following theorems.
\begin{thm}
	System \eqref{1.1} admits a unique disease-free equilibrium given by $E^0=(S^0,V^0,0,0)$, where
	\[
	S^0=\frac{\Lambda ( d_0+\upsilon)}{ d_0( d_0+\varpi+\upsilon)},\qquad
	V^0=\frac{\varpi\Lambda }{ d_0( d_0+\varpi+\upsilon)}.
	\]
\end{thm}
\begin{proof}

At an equilibrium of system \eqref{1.1}, all time derivatives vanish. Moreover, since an equilibrium is time-independent, we have \(I(t-\tau)=I\). Therefore, an equilibrium \((S,V,I,R)\) satisfies \begin{equation}\label{2.1} 
\begin{cases} \Lambda-\iota e^{-mI}IS-(d_0+\varpi)S+\upsilon V+\vartheta R=0,\\ \varpi S-\sigma_0\iota e^{-mI}IV-(d_0+\upsilon)V=0,\\ \iota e^{-mI}I(S+\sigma_0V)-(d_0+\gamma_0+d)I=0,\\ \dfrac{\gamma_0I}{1+\alpha I}-(d_0+\vartheta)R=0.
 \end{cases}
 \end{equation} For a disease-free equilibrium, we set \(I=0\). The fourth equation of \eqref{2.1} then gives \(R=0\). Consequently, the first two equations reduce to \[ \Lambda-(d_0+\varpi)S+\upsilon V=0, \qquad \varpi S-(d_0+\upsilon)V=0. \] Solving this linear system yields \[ S^0= \frac{\Lambda(d_0+\upsilon)} {d_0(d_0+\varpi+\upsilon)}, \qquad V^0= \frac{\varpi\Lambda} {d_0(d_0+\varpi+\upsilon)}. \] Hence, system \eqref{1.1} admits the unique disease-free equilibrium \[ E^0= \left( \frac{\Lambda(d_0+\upsilon)} {d_0(d_0+\varpi+\upsilon)}, \frac{\varpi\Lambda} {d_0(d_0+\varpi+\upsilon)}, 0,0 \right). \]	
		
\end{proof}

\begin{thm}
	Assume that $\mathcal{R}_0>1$ and
	\[
	m \geq \frac{\sigma_0 \iota \varpi\left(1-\sigma_0\right)}{\left(d_0+\upsilon+\varpi\right)\left(d_0+\upsilon+\sigma_0 \varpi\right)}.
	\]
	Then system \eqref{1.1} admits a unique endemic equilibrium
	\[
	E^*=\left(S^*, V^*, I^*, R^*\right),
	\]
	where
	\[
	\mathcal{R}_0=\frac{\Lambda \iota\left(d_0+\upsilon+\sigma_0 \varpi\right)}{d_0\left(d_0+\varpi+\upsilon\right)\left(d_0+\gamma_0+d\right)}.\]
\end{thm}

\begin{proof}
	Let
	\[
	E^*=\left(S^*, V^*, I^*, R^*\right)
	\]
	be an endemic equilibrium of system \eqref{1.1}. Since $I^*>0$, the equilibrium components satisfy
	\begin{equation}\label{2.2}
		\left\{\begin{array}{l}
			\Lambda-\iota I^* e^{-m I^*} S^*-\left(d_0+\varpi\right) S^*+\upsilon V^*+\vartheta R^*=0, \\
			\varpi S^*-\sigma_0 \iota I^* e^{-m I^*} V^*-\left(d_0+\upsilon\right) V^*=0, \\
			\iota I^* e^{-m I^*}\left(S^*+\sigma_0 V^*\right)-\left(d_0+\gamma_0+d\right) I^*=0, \\
			\frac{\gamma_0 I^*}{1+\alpha I^*}-\left(d_0+\vartheta\right) R^*=0 .
		\end{array}\right.
	\end{equation}
	
	For convenience, denote
	\[
	K=d_0+\gamma_0+d, \quad a=d_0+\upsilon.
	\]
	Summing all four equations in \eqref{2.2}, internal transition terms cancel, and we obtain
	\[
	\Lambda-d_0\left(S^*+V^*+R^*\right)-K I^*+\frac{\gamma_0 I^*}{1+\alpha I^*}=0 .
	\]
	From the fourth equation of \eqref{2.2},
	\[
	R^*=\frac{\gamma_0 I^*}{\left(d_0+\vartheta\right)\left(1+\alpha I^*\right)}.
	\]
	Substituting this expression into the above identity yields
	\[
	S^*+V^*=F\left(I^*\right),
	\]
	where
	\begin{equation}\label{2.3}
		F(I)=\frac{\Lambda}{d_0}-\frac{K}{d_0} I+\frac{\gamma_0 \vartheta I}{d_0\left(d_0+\vartheta\right)(1+\alpha I)} .
	\end{equation}
	
	We now derive a second expression for $S^*+V^*$. Since $I^*>0$, dividing the third equation of \eqref{2.2} by $I^*$ gives
	\begin{equation}\label{2.4}
		S^*+\sigma_0 V^*=\frac{K}{\iota} e^{m I^*}.
	\end{equation}
	From the second equilibrium equation, we solve for $V^*$:
	\[
	V^*=\frac{\varpi S^*}{a+\sigma_0 \iota I^* e^{-m I^*}}.
	\]
	Define
	\[
	D(I)=a+\sigma_0 \iota I e^{-m I}.
	\]
	Then
	\[
	V^*=\frac{\varpi}{D\left(I^*\right)} S^*.
	\]
	Substituting into \eqref{2.4}, we obtain
	\[
	S^*=\frac{K e^{m I^*}}{\iota} \cdot \frac{D\left(I^*\right)}{D\left(I^*\right)+\sigma_0 \varpi},
	\]
	and hence
	\[
	V^*=\frac{K e^{m I^*}}{\iota} \cdot \frac{\varpi}{D\left(I^*\right)+\sigma_0 \varpi}.
	\]
	Therefore,
	\[
	S^*+V^*=G\left(I^*\right),
	\]
	where
	\begin{equation}\label{2.5}
		G(I)=\frac{K e^{m I}}{\iota} \cdot \frac{D(I)+\varpi}{D(I)+\sigma_0 \varpi}.
	\end{equation}
	Consequently, the infected component $I^*$ of any endemic equilibrium must satisfy
	\begin{equation} \label{2.6}
		F\left(I^*\right)=G\left(I^*\right).
	\end{equation}
	
	We next analyze the monotonicity of $F$ and $G$. Differentiating \eqref{2.3},
	\[
	F^{\prime}(I)=-\frac{K}{d_0}+\frac{\gamma_0 \vartheta}{d_0\left(d_0+\vartheta\right)(1+\alpha I)^2}.
	\]
	Since
	\[
	\frac{\vartheta}{\left(d_0+\vartheta\right)(1+\alpha I)^2}<1,
	\]
	we have
	\[
	F^{\prime}(I)<-\frac{K}{d_0}+\frac{\gamma_0}{d_0}=-\frac{d_0+d}{d_0}<0 .
	\]
	Hence, $F(I)$ is strictly decreasing on $[0,+\infty)$.
	
	Now consider $G(I)$. Direct computation yields
	\[
	D^{\prime}(I)=\sigma_0 \iota e^{-m I}(1-m I).
	\]
	Taking the logarithmic derivative of $G(I)$, we obtain
	\begin{equation} \label{2.7}
		\frac{G'(I)}{G(I)}
		=m+\frac{\varpi(\sigma_0-1) D'(I)}{\big(D(I)+\varpi\big)\big(D(I)+\sigma_0 \varpi\big)}.
	\end{equation}
	
	For $0 \leq I \leq 1/m$, one has
	$$
	0 \leq D^{\prime}(I) \leq \sigma_0 \iota,
	\qquad
	D(I) \geq a.
	$$
	Therefore,
	$$
	\frac{\varpi\left(1-\sigma_0\right) D^{\prime}(I)}{(D(I)+\varpi)\left(D(I)+\sigma_0 \varpi\right)} \leq \frac{\sigma_0 \iota \varpi\left(1-\sigma_0\right)}{(a+\varpi)\left(a+\sigma_0 \varpi\right)}.
	$$
	Combined with the condition imposed on $m$, it follows from \eqref{2.7} that
	$$
	G^{\prime}(I) \ge 0,\quad 0 \leq I \leq \frac{1}{m}.
	$$
	
	For $I \geq 1/m$, we have $D^{\prime}(I) \leq 0$. Recall $\sigma_0\in[0,1]$ from the model description, so $\sigma_0-1\le0$. We deduce
	\[
	\frac{\varpi(\sigma_0-1) D'(I)}{\big(D(I)+\varpi\big)\big(D(I)+\sigma_0\varpi\big)}\ge 0,
	\]
	which implies
	\[
	\frac{G'(I)}{G(I)}\ge m>0.
	\]
	Thus $G'(I)>0$.
	
	Combining the two cases above, $G'(I)>0$ holds for all $I\ge0$. Therefore, $G(I)$ is strictly increasing on $[0,+\infty)$.
	
	We now prove that equation \eqref{2.6} possesses exactly one positive solution. At $I=0$,
	\[
	F(0)=\frac{\Lambda}{d_0},
	\qquad
	G(0)=\frac{K}{\iota} \cdot \frac{a+\varpi}{a+\sigma_0 \varpi}.
	\]
	By definition,
	\[
	\mathcal{R}_0=\frac{\Lambda \iota\left(a+\sigma_0 \varpi\right)}{d_0(a+\varpi) K}=\frac{F(0)}{G(0)}.
	\]
	Hence $\mathcal{R}_0>1$ implies $F(0)>G(0)$.
	
	Furthermore,
	\[
	\lim _{I \rightarrow+\infty} F(I)=-\infty,
	\qquad
	\lim _{I \rightarrow+\infty} G(I)=+\infty.
	\]
	Since $F$ and $G$ are continuous on $[0,\infty)$, there exists at least one $I^*>0$ satisfying \eqref{2.6}. Moreover, $F$ is strictly decreasing and $G$ is strictly increasing, so such a positive solution $I^*$ is unique.
	
	Once $I^*$ is uniquely determined, the remaining equilibrium components are uniquely given by
	$$
	\begin{aligned}
		S^* & =\frac{K e^{m I^*}}{\iota} \cdot \frac{D\left(I^*\right)}{D\left(I^*\right)+\sigma_0 \varpi}, \\
		V^* & =\frac{K e^{m I^*}}{\iota} \cdot \frac{\varpi}{D\left(I^*\right)+\sigma_0 \varpi},
	\end{aligned}
	$$
	and
	$$
	R^*=\frac{\gamma_0 I^*}{\left(d_0+\vartheta\right)\left(1+\alpha I^*\right)}.
	$$
	Therefore, system \eqref{1.1} admits a unique endemic equilibrium
	$$
	E^*=\left(S^*, V^*, I^*, R^*\right).
	$$
	The proof is complete.
\end{proof}

\section{Local Asymptotic Stability}
\subsection{Local Asymptotic Stability of $E^0$}
We next investigate the local asymptotic stability of the disease-free equilibrium
$E^0$.

\begin{thm}
	The disease-free equilibrium $E^0$ is locally asymptotically stable
	for $\mathcal{R}_0<1$ and unstable for $\mathcal{R}_0>1$.
\end{thm}

\begin{proof}
	By linearizing system (1.1) at the disease-free equilibrium $E^0 = (S^0, V^0, 0, 0)$, we derive the characteristic equation
	$$
	(\lambda+ d_0+\vartheta)(\lambda+ d_0)(\lambda+ d_0+\upsilon+\varpi)g_0(\lambda)=0,
	$$
	where
	$$
	g_0(\lambda)
	=
	\lambda+ d_0+\gamma_0 +d-\iota(S^0+\sigma_0 V^0).
	$$
	
	It is clear that 
	$\lambda_{01}=- d_0$, 
	$\lambda_{02}=- d_0-\vartheta$, 
	and 
	$\lambda_{03}=- d_0-\upsilon-\varpi$
	are all negative real eigenvalues. 
	
	The remaining eigenvalue is obtained from 
	$g_0(\lambda)=0$, which yields
	$$
	\lambda_{04}
	=
	\iota(S^0+\sigma_0 V^0)-( d_0+\gamma_0 +d).
	$$
Using the expression for $\mathcal{R}_0$, we obtain
	$$
	\lambda_{04}
	=
	( d_0+\gamma_0 +d)(\mathcal{R}_0-1).
	$$
		
Therefore, if $\mathcal R_0<1$, then $\lambda_{04}<0$, and all eigenvalues
have negative real parts. Hence, $E^0$ is locally asymptotically stable.

Conversely, if $\mathcal R_0>1$, then $\lambda_{04}>0$, and therefore
$E^0$ is unstable.	
	
\end{proof}
\subsection{Local Asymptotic Stability of $E^*$}
When \(\mathcal{R}_0 > 1\) and $m \geq \frac{\sigma_0 \iota \varpi\left(1-\sigma_0\right)}{\left(d_0+\upsilon+\varpi\right)\left(d_0+\upsilon+\sigma_0 \varpi\right)}$, a positive equilibrium $E^*$ exists for system \eqref{1.1} and fulfills the equilibrium equations in \eqref{2.2}. For convenience, let
\[
A=\iota e^{-mI^*}, \qquad
B=S^*+\sigma_0 V^*, \qquad
C=d_0+\gamma_0+d.
\]
Then one obtains the identity
\[
A\cdot B=C.
\]

Linearizing system \eqref{1.1} at the equilibrium point \(E^*\) leads to the following characteristic equation:
\[
\lambda^4+a_{11}\lambda^3+a_{12}\lambda^2+a_{13}\lambda+a_{14}
+(a_{21}\lambda^3+a_{22}\lambda^2+a_{23}\lambda+a_{24})e^{-\lambda\tau}=0.
\]
where
\begin{align*}
	a_{11}
	&=A I^*+d_0+\varpi+\sigma_0 A I^*+d_0+\upsilon+d_0+\vartheta,\\
	a_{12}
	&=\bigl(A I^*+d_0+\varpi\bigr)\bigl(\sigma_0 A I^*+d_0+\upsilon\bigr)
	+\sigma_0^2 A^2 V^* I^*
	-\upsilon\varpi
	+A^2 S^* I^*\\
	&\quad+\bigl(d_0+\vartheta\bigr)\bigl(A I^*+d_0+\varpi+\sigma_0 A I^*+d_0+\upsilon\bigr),\\
	a_{13}
	&=\bigl(d_0+\vartheta\bigr)\Big\{
	\bigl(A I^*+d_0+\varpi\bigr)\bigl(\sigma_0 A I^*+d_0+\upsilon\bigr)
	+\sigma_0^2 A^2 V^* I^*
	-\upsilon\varpi
	+A^2 S^* I^*
	\Big\}\\
	&\quad+\bigl(A I^*+d_0+\varpi\bigr)\sigma_0^2 A^2 V^* I^*
	+\upsilon\sigma_0 A^2 V^* I^*
	+\sigma_0\varpi A^2 S^* I^*
	+A^2 S^* I^*\bigl(\sigma_0 A I^*+d_0+\upsilon\bigr)\\
	&\quad-\frac{\gamma_0}{(1+\alpha I^*)^2}\vartheta A I^*,\\
	a_{14}
	&=\bigl(d_0+\vartheta\bigr)\Big\{
	\sigma_0^2 A^2 V^* I^*\bigl(A I^*+d_0+\varpi\bigr)
	+\upsilon\sigma_0 A^2 V^* I^*
	+\sigma_0\varpi A^2 S^* I^*
	+A^2 S^* I^*\bigl(\sigma_0 A I^*+d_0+\upsilon\bigr)
	\Big\}\\
	&\quad-\frac{\gamma_0}{(1+\alpha I^*)^2}\vartheta\Big\{
	\sigma_0\varpi A I^*+A I^*\bigl(\sigma_0 A I^*+d_0+\upsilon\bigr)
	\Big\},\\
	a_{21}
	&=m C I^*,\\
	a_{22}
	&=m C I^*\bigl(A I^*+d_0+\varpi+\sigma_0 A I^*+d_0+\upsilon+d_0+\vartheta\bigr)
	-m\sigma_0^2 A^2 V^* (I^*)^2
	-m A^2 S^* (I^*)^2,\\
	a_{23}
	&=m C I^*(d_0+\vartheta)\bigl(A I^*+d_0+\varpi+\sigma_0 A I^*+d_0+\upsilon\bigr)
	-(d_0+\vartheta)\bigl(m\sigma_0^2 A^2 V^* (I^*)^2+m A^2 S^* (I^*)^2\bigr)\\
	&\quad
	+m C I^*\bigl(A I^*+d_0+\varpi\bigr)\bigl(\sigma_0 A I^*+d_0+\upsilon\bigr)
	-m\sigma_0^2 A^2 V^* (I^*)^2\bigl(A I^*+d_0+\varpi\bigr)
	-m C I^*\upsilon\varpi\\
	&\quad-m\sigma_0 A^2 V^* (I^*)^2\upsilon
	-m\sigma_0\varpi A^2 S^* (I^*)^2
	-m A^2 S^* (I^*)^2\bigl(\sigma_0 A I^*+d_0+\upsilon\bigr),\\
	a_{24}
	&=(d_0+\vartheta)\Big\{
	m C I^*\bigl(A I^*+d_0+\varpi\bigr)\bigl(\sigma_0 A I^*+d_0+\upsilon\bigr)
	-m\sigma_0^2 A^2 V^* (I^*)^2\bigl(A I^*+d_0+\varpi\bigr)
	-m C I^*\upsilon\varpi\\
	&\quad-m\sigma_0 A^2 V^* (I^*)^2\upsilon
	-m\sigma_0\varpi A^2 S^* (I^*)^2
	-m A^2 S^* (I^*)^2\bigl(\sigma_0 A I^*+d_0+\upsilon\bigr)
	\Big\}.
\end{align*}

For $\tau=0$, the characteristic equation reduces to
\begin{equation}\label{3.1}
	\lambda^{4}
	+(a_{11}+a_{21})\lambda^{3}
	+(a_{12}+a_{22})\lambda^{2}
	+(a_{13}+a_{23})\lambda
	+(a_{14}+a_{24})=0.
\end{equation}

According to the Routh--Hurwitz stability criterion, all roots of the characteristic equation \eqref{3.1} have negative real parts provided that the following inequalities are satisfied:

\begin{equation}
	\begin{cases}
		a_{11}+a_{21}>0,\quad 
		a_{12}+a_{22}>0,\quad 
		a_{13}+a_{23}>0,\quad 
		a_{14}+a_{24}>0,\\
		
		(a_{11}+a_{21})(a_{12}+a_{22})>a_{13}+a_{23},\\
		
		\left(a_{11}+a_{21}\right)\left(a_{12}+a_{22}\right)\left(a_{13}+a_{23}\right)>\left(a_{13}+a_{23}\right)^2+\left(a_{11}+a_{21}\right)^2\left(a_{14}+a_{24}\right).
	\end{cases}
	\tag{H}
	\label{eq:H}
\end{equation}

Therefore, we obtain the following result.
\begin{thm}
	Assume that \(\mathcal R_0>1\), $m \geq \frac{\sigma_0 \iota \varpi\left(1-\sigma_0\right)}{\left(d_0+\upsilon+\varpi\right)\left(d_0+\upsilon+\sigma_0 \varpi\right)}$ and condition \eqref{eq:H} holds.
	Then, for \(\tau=0\), the endemic equilibrium
	\(E^*=(S^*,V^*,I^*,R^*)\) of system \eqref{1.1}
	is locally asymptotically stable.
\end{thm}
\section{Global Asymptotic Stability}

The feasible region for instantaneous population states is defined as
\[
\Omega = \left\{ (S,V,I,R)\in\mathbb R_+^4: S+V+I+R\leq\frac{\Lambda}{d_0} \right\}.
\]
We further introduce the set of admissible initial functions
\[
X=\big\{\psi\in C_+ \,\big|\, \psi(\theta)\in\Omega,\quad \forall\,\theta\in[-\tau,0]\big\},
\]
which collects all non-negative initial histories satisfying the population bound at every historical moment.

\subsection{Global asymptotic stability of the disease-free equilibrium $E^0$}

\begin{thm}
	Assume that 
	\[
	\frac{\iota\Lambda}{d_0(d_0+\gamma_0+d)}<1.
	\]
	Then the disease-free equilibrium
	$E^0=(S^0,V^0,0,0)$ of system \eqref{1.1}
	is globally asymptotically stable in $\Omega$. 
\end{thm}

\begin{proof}
	Let $N(t)=S(t)+V(t)+I(t)+R(t)$.
	It has been verified that the feasible region $\Omega$ is positively invariant and all solutions are ultimately bounded, with
	\[
	\limsup_{t\rightarrow\infty}N(t)
	\leq \frac{\Lambda}{d_0}.
	\]
	Since $0\le\sigma_0\le 1$, we have
	\[
	S(t)+\sigma_0V(t)
	\leq S(t)+V(t)
	\leq N(t)
	\leq \frac{\Lambda}{d_0}.
	\]
	
	Consider the infected compartment of system \eqref{1.1}.
	Using the estimate $e^{-mI(t-\tau)}\leq1$, we derive
	\[
	\begin{aligned}
		\frac{dI(t)}{dt}
		&=
		\iota I(t)e^{-mI(t-\tau)}
		\big(S(t)+\sigma_0V(t)\big)
		-(d_0+\gamma_0+d)I(t)
		\\
		&\leq
		\left[
		\frac{\iota\Lambda}{d_0}
		-(d_0+\gamma_0+d)
		\right]I(t).
	\end{aligned}
	\]
	Set
	\[
	\delta:=
	(d_0+\gamma_0+d)
	-\frac{\iota\Lambda}{d_0}>0.
	\]
	Then
	\[
	\frac{dI(t)}{dt}\leq-\delta I(t).
	\]
	Combined with nonnegativity $I(t)\ge0$, one obtains
	\[
	\lim_{t\rightarrow\infty}I(t)=0.
	\]
	
	As $I(t)\to0$, the original functional differential system is asymptotically autonomous, whose limiting system is obtained by formally setting $I\equiv0$:
	\[
	\begin{cases}
		\dot S
		=
		\Lambda-(d_0+\varpi)S+\upsilon V+\vartheta R,\\
		\dot V
		=
		\varpi S-(d_0+\upsilon)V,\\
		\dot R
		=
		-(d_0+\vartheta)R.
	\end{cases}
	\]
	The unique equilibrium of this limiting subsystem is $(S^0,V^0,0)$, where
	\[
	S^0=
	\frac{\Lambda(d_0+\upsilon)}
	{d_0(d_0+\varpi+\upsilon)},
	\qquad
	V^0=
	\frac{\varpi\Lambda}
	{d_0(d_0+\varpi+\upsilon)}.
	\]
	
	The coefficient matrix associated with the linearized limiting subsystem reads
	\[
	A=
	\begin{pmatrix}
		-(d_0+\varpi)&\upsilon&\vartheta\\
		\varpi&-(d_0+\upsilon)&0\\
		0&0&-(d_0+\vartheta)
	\end{pmatrix}.
	\]
	Clearly, one eigenvalue is
	\[
	\lambda_1=-(d_0+\vartheta)<0.
	\]
	The remaining two eigenvalues satisfy the quadratic equation
	\[
	\lambda^2+
	(2d_0+\varpi+\upsilon)\lambda
	+d_0(d_0+\varpi+\upsilon)=0.
	\]
	By the Routh--Hurwitz criterion, all roots have negative real parts.
	Therefore, the limiting subsystem is globally asymptotically stable at $(S^0,V^0,0)$.
	
	By the theory of asymptotically autonomous functional differential equations, every $\omega$-limit point of any forward solution of the original system belongs to the invariant set of the limiting system (Theorem 4.1 in \cite{thieme1992convergence}).
	Since all solutions of the limiting subsystem converge to $(S^0,V^0,0)$, we conclude
	\[
	\lim_{t\rightarrow\infty}S(t)=S^0,\quad
	\lim_{t\rightarrow\infty}V(t)=V^0,\quad
	\lim_{t\rightarrow\infty}R(t)=0.
	\]
	Together with $\lim_{t\to\infty}I(t)=0$, we arrive at
	\[
	\lim_{t\rightarrow\infty}
	\big(S(t),V(t),I(t),R(t)\big)
	=
	\big(S^0,V^0,0,0\big)=E^0.
	\]
	Hence, the disease-free equilibrium $E^0$ is globally asymptotically stable in $\Omega$.
\end{proof}

\newtheorem{lemma}{\bf Lemma}[section]

\subsection{Uniform persistence of positive solutions}

We now investigate the uniform persistence of positive solutions in the case where $\mathcal{R}_0>1$. The following lemma will be needed in the subsequent analysis.

Denote by $u(t,\psi)=(S(t,\psi),V(t,\psi),I(t,\psi),R(t,\psi))$ the unique solution of system \eqref{1.1} satisfying the initial function $\psi \in X$.
\begin{lem}
	Let $\psi\in X$. If there exists some $t^*\ge0$ such that $I(t^*,\psi)>0$, then
	\[
	S(t,\psi)>0,\quad V(t,\psi)>0,\quad I(t,\psi)>0,\quad R(t,\psi)>0,\qquad \forall\,t>t^*.
	\]
\end{lem}

\begin{proof}
	
	We prove the strict positivity componentwise via contradiction arguments.
	
	(1) Positivity of $I(t)$. From the third equation of \eqref{1.1},
	\[
	\frac{dI(t)}{dt} = \iota I(t)e^{-mI(t-\tau)}\big(S(t) + \sigma_0 V(t)\big) - ( d_0 + \gamma_0  + d)I(t).
	\]
	Suppose, for contradiction, that there exists $t_1>t^*$ such that
	\[
	I(t_1,\psi)=0,\qquad I(t,\psi)>0,\quad \forall\,t\in(t^*,t_1).
	\]
	For all $t\in[t^*,t_1]$, we have the lower estimate
	\[
	\frac{dI(t)}{dt}
	\geq
	-\big( d_0+\gamma_0 +d\big)I(t).
	\]
	Applying the comparison principle,
	\[
	I(t_1,\psi) \geq I(t^*,\psi)\exp\big\{-( d_0 + \gamma_0  + d)(t_1-t^*)\big\} > 0,
	\]
	which contradicts $I(t_1,\psi) = 0$. Consequently, $I(t,\psi) > 0$ for all $t > t^*$.
	
	(2) Positivity of $S(t)$. Consider the first equation
	\[
	\frac{dS(t)}{dt} = \Lambda  - \iota I(t)e^{-mI(t-\tau)}S(t) - ( d_0 + \varpi)S(t) + \upsilon V(t) + \vartheta R(t).
	\]
	Assume there exists the first time $t_2 > t^*$ such that $S(t_2,\psi) = 0$ and $S(t,\psi) > 0$ for all $t \in (t^*, t_2)$. On the interval $(t^*,t_2)$,
	\[
	\frac{dS(t)}{dt} \geq -\bigl(\iota I(t)e^{-mI(t-\tau)} +  d_0 + \varpi\bigr)S(t).
	\]
	Let
	\[
	c_1 = \sup_{t\in[t^*,t_2]}\big\{\iota I(t)e^{-mI(t-\tau)}\big\} +  d_0 + \varpi <+\infty.
	\]
	For any fixed small $\delta\in(0,t_2-t^*)$, we have $S(t^*+\delta)>0$ and
	\[
	S(t_2,\psi) \geq S(t^*+\delta)\exp\big\{-c_1(t_2-t^*-\delta)\big\} > 0,
	\]
	which contradicts $S(t_2,\psi)=0$. Hence $S(t,\psi) > 0$ for all $t > t^*$.
	
	(3) Positivity of $V(t)$. The second equation reads
	\[
	\frac{dV(t)}{dt} = \varpi S(t) - \sigma_0 \iota I(t) e^{-mI(t-\tau)} V(t) - ( d_0 + \upsilon) V(t).
	\]
	Suppose there exists the first time $t_3 > t^*$ satisfying $V(t_3,\psi) = 0$ and $V(t,\psi) > 0$ for every $t \in (t^*, t_3)$. For $t\in(t^*,t_3)$,
	\[
	\frac{dV(t)}{dt} \geq -\bigl( \sigma_0 \iota I(t) e^{-mI(t-\tau)} +  d_0 + \upsilon \bigr) V(t).
	\]
	Define
	\[
	c_2 = \sup_{t\in[t^*,t_3]}\big\{\sigma_0 \iota I(t) e^{-mI(t-\tau)}\big\} +  d_0 + \upsilon <+\infty.
	\]
	Take small $\delta>0$, then $V(t^*+\delta)>0$ and
	\[
	V(t_3,\psi) \geq V(t^*+\delta)\exp\big\{-c_2(t_3-t^*-\delta)\big\} > 0,
	\]
	contradicting $V(t_3,\psi)=0$. Therefore, $V(t,\psi) > 0$ for all $t > t^*$.
	
	(4) Positivity of $R(t)$. From the fourth equation,
	\[
	\frac{dR(t)}{dt} = \frac{\gamma_0  I(t)}{1 + \alpha I(t)} - ( d_0 + \vartheta) R(t).
	\]
	Assume for contradiction that there exists the first time $t_4 > t^*$ such that $R(t_4,\psi) = 0$ and $R(t,\psi) > 0$ for all $t \in (t^*, t_4)$. For $t\in(t^*,t_4)$,
	\[
	\frac{dR(t)}{dt} \geq -\big( d_0 + \vartheta\big) R(t).
	\]
	For small $\delta>0$, $R(t^*+\delta)>0$, and the comparison principle yields
	\[
	R(t_4,\psi) \geq R(t^*+\delta)\exp\big\{-( d_0 + \vartheta)(t_4-t^*-\delta)\big\} > 0,
	\]
	which contradicts $R(t_4,\psi)=0$. Thus $R(t,\psi) > 0$ for all $t > t^*$.
	
	Summarizing the four parts, all components satisfy $S(t,\psi)>0$, $V(t,\psi)>0$, $I(t,\psi)>0$, $R(t,\psi)>0$ for every $t>t^*$. The proof is completed.
\end{proof}

\begin{thm}
	Suppose $\mathcal{R}_0>1$ and
	\[
	m \geq \frac{\sigma_0 \iota \varpi\left(1-\sigma_0\right)}{\left(d_0+\upsilon+\varpi\right)\left(d_0+\upsilon+\sigma_0 \varpi\right)}.
	\]
There exists a positive constant $\upsilon_0$ such that every solution
	\[
	u(t,\psi)
	=
	\big(S(t,\psi),V(t,\psi),I(t,\psi),R(t,\psi)\big),
	\]
	with initial function $\psi\in X $ and $\psi_3(0)>0$, satisfies
	\begin{equation}\label{4.1}
		\liminf_{t\to\infty}
		\min
		\big\{
		S(t,\psi),
		V(t,\psi),
		I(t,\psi),
		R(t,\psi)
		\big\}
		\geq
		\upsilon_0.
	\end{equation}
\end{thm}

\begin{proof}

	Define the sets
	\[
	X^0
	=
	\left\{
	\psi\in X :
	\psi_3(0)>0
	\right\},\quad
	\partial X ^0=X\setminus X^0
	=
	\left\{
	\psi\in X :
	\psi_3(0)=0
	\right\}.
	\]
	According to system \eqref{1.1}, it is straightforward to verify that both
	$X$ and $X^0$ are positively invariant.
In addition, $\partial X^0$ is relatively closed with respect to $X$.

Let $T(t)$ denote the solution semiflow generated by system \eqref{1.1}. For any solution $u(\cdot,\psi)$ and $t\ge0$, define the segment $u_t\in C$ by
\[
u_t(\theta)=u(t+\theta,\psi),\quad \theta\in[-\tau,0].
\]
Then the semiflow is given by
\[
T(t)\psi
=
u_t(\cdot,\psi),
\qquad
\forall\, t\ge0,\ \psi\in X.
\]

According to \cite[Theorem 3.6.1]{hale1993introduction}, the semiflow $T(t)$ enjoys continuity and compactness for all $t > \tau$. Since every solution of system \eqref{1.1} is ultimately bounded, the semiflow $T(t)$ is point dissipative. Combining these properties and \cite[Theorem 3.4.8]{1988Asymptotic}, the semiflow $T(t)$ possesses a global attractor denoted by $K$.

Define the set
\[
X_\partial = \{\psi \in \partial X^0 : T(t)\psi \in \partial X^0 ,\ \forall\, t \geq 0\}.
\]
For system \eqref{1.1}, any initial  $\psi\in \partial X^0$  implies $I(t,\psi) = 0$ for all $t\ge0$. Thus, the forward trajectory always remains in $\partial X^0$, yielding $X_\partial=\partial X^0$.

We now establish the dynamical property of the semiflow on the  set $X_\partial$.

\noindent\textbf{Claim 1.}
{\it The disease-free equilibrium $E^0$ is globally asymptotically stable with respect to the semiflow $T(t)$ restricted to $X_\partial$.}

	For any $\psi\in X_\partial$, we have $I(t,\psi)\equiv 0$ for all $t\ge0$. Substituting $I\equiv 0$ into \eqref{1.1}, the system reduces to the following boundary subsystem:

 	\[
 	\begin{cases}
 		\dfrac{dS(t)}{dt}
 		=
 		\Lambda -( d_0+\varpi)S(t)+\upsilon V(t)+\vartheta R(t), \\[6pt]
 		
 		\dfrac{dV(t)}{dt}
 		=
 		\varpi S(t)-( d_0+\upsilon)V(t), \\[6pt]
 		
 		\dfrac{dR(t)}{dt}
 		=
 		-( d_0+\vartheta)R(t).
 	\end{cases}
 	\]
	
	It is obvious that $R(t)$ decays exponentially to zero, i.e., $\lim_{t\to\infty}R(t,\psi)=0$. As a result, the long-term dynamics of $S(t)$ and $V(t)$ are asymptotically governed by a stable two-dimensional linear subsystem, which possesses a unique globally attractive nonnegative equilibrium $(S^0,V^0)$. Therefore, all trajectories starting from $X_\partial$ satisfy
	\[
	\lim_{t\to\infty}S(t,\psi)=S^0,
	\qquad
	\lim_{t\to\infty}V(t,\psi)=V^0,
	\qquad
	\lim_{t\to\infty}R(t,\psi)=0.
	\]
	This indicates that $E^0$ is globally attractive on $X_\partial$.

	Furthermore, on the invariant set $X_{\partial}$, system \eqref{1.1} reduces to an autonomous ordinary differential system.
	By the linearization stability criterion for ordinary differential equations, the equilibrium $E^0$ is locally asymptotically stable with respect to $T(t)$ on $X_\partial$.  Combining global attractivity and local asymptotic stability, we conclude that $E^0$ is globally asymptotically stable on $X_\partial$.
	This completes the proof of Claim 1.

	Since $\mathcal{R}_0 > 1$, one can choose a sufficiently small constant $\varepsilon > 0$ such that
	\[
	\mathcal{R}_0^{\epsilon}
	=
	\frac{
		\iota e^{-m\epsilon}
		\left(
		(S^0-\epsilon)+\sigma_0(V^0-\epsilon)
		\right)
	}{
		d_0+\gamma_0 +d
	}
	>1.
	\]

\noindent\textbf{Claim 2.}
{\it For every $\psi\in X^0$,
\[
\limsup_{t\to\infty}
\|T(t)\psi-E^0\|
\geq
\epsilon.
\] }

Suppose on the contrary that the assertion fails.
Then there exists some $\psi\in\ X^0$ such that
\[
\limsup_{t\to\infty}
\|T(t)\psi-E^0\|
<
\epsilon.
\]
We choose sufficiently small $\epsilon>0$ satisfying $S^0-\epsilon>0$ and $V^0-\epsilon>0$.
By the definition of $\limsup$, there exists $\bar{t}>0$ such that
\[
\|T(t)\psi-E^0\|<\epsilon,\quad \forall\, t\geq\bar{t}.
\]
Recall that $T(t)\psi=u_t$ is the solution segment satisfying
\[
u_t(\theta)=\big(S(t+\theta,\psi),V(t+\theta,\psi),I(t+\theta,\psi),R(t+\theta,\psi)\big),\quad \theta\in[-\tau,0].
\]
Since $E^0$ is the constant function $E^0(\theta)=(S^0,V^0,0,0)$, we obtain
\[
|S(t+\theta,\psi)-S^0|<\epsilon,\quad
|V(t+\theta,\psi)-V^0|<\epsilon,\quad
0<I(t+\theta,\psi)<\epsilon,\quad
0<R(t+\theta,\psi)<\epsilon,\quad
\forall\theta\in[-\tau,0],\;\forall t\geq\bar{t}.
\]
In particular, taking $\theta=0$, we have
\[
S(t)>S^0-\epsilon,\quad V(t)>V^0-\epsilon,\quad 0<I(t)<\epsilon,
\quad \forall\, t\geq \bar{t}.
\]

Therefore, the third equation of system \eqref{1.1} yields
\[
\frac{dI(t)}{dt}
\geq
\iota I(t)e^{-m\epsilon}
\big[
(S^0-\epsilon)
+\sigma_0(V^0-\epsilon)
\big]
-( d_0+\gamma_0 +d)I(t).
\]

Define
\[
c
=
\iota e^{-m\epsilon}
\big[
(S^0-\epsilon)
+\sigma_0(V^0-\epsilon)
\big]
-( d_0+\gamma_0 +d).
\]
Since $\mathcal{R}_0^{\epsilon}>1$, we have $c>0$.

Applying the comparison principle, we obtain
\[
I(t)
\geq
I(\bar{t}+\tau)
e^{c\bigl(t-(\bar{t}+\tau)\bigr)},
\qquad
t\geq\bar{t}+\tau.
\]
Hence, $I(t)\to+\infty$ as $t\to\infty$, which contradicts $I(t)<\epsilon$.

Therefore,
\[
\limsup_{t\to\infty}
\|T(t)\psi-E^0\|
\geq
\epsilon,
\qquad
\forall\, \psi\in X_0.
\]

Claim 1 implies that the boundary invariant set $X_\partial$ is acyclic and every trajectory in $X_\partial$ converges to $E^0$. Meanwhile, Claim 2 shows that $E^0$ is an isolated invariant set in $X$, and its stable set $W^s(E^0)$ satisfies $W^s(E^0)\cap X^0=\emptyset$, where $W^s(E^0)$ denotes the stable set associated with the semiflow $T(t)$.
Therefore, by the acyclicity criterion for uniform persistence of autonomous dynamical semiflows \cite[Theorem 1.3.1, Remarks 1.3.1–1.3.2]{zhao2017dynamical}, we conclude that the semiflow $T(t): X\to X$ is uniformly persistent with respect to the pair $(X^0,\partial X^0)$.

Furthermore, according to  \cite[Theorem 2.4]{1995Uniform}, the semiflow $T(t)$ admits a global attractor $A_0\subset X^0$. The uniform persistence property implies that this attractor contains an interior invariant point. Consequently, system \eqref{1.1} admits a positive coexistence steady state $\bar{\psi}=(\bar{\psi}_1,\bar{\psi}_2,\bar{\psi}_3,\bar{\psi}_4)\in X_0$ satisfying $T(t)\bar{\psi}=\bar{\psi}$ for all $t\geq0$.
Since such steady state is time-independent, we set $S^*=\bar{\psi}_1(0)$, $V^*=\bar{\psi}_2(0)$, $I^*=\bar{\psi}_3(0)$ and $R^*=\bar{\psi}_4(0)$. Hence,
$
E^*=(S^*,V^*,I^*,R^*)
$
is a strictly positive endemic equilibrium of system \eqref{1.1}.

Finally, we verify uniform persistence for all state variables by introducing a continuous functional $p:X\to\mathbb{R}_+$ defined by
\[
p(\psi)=\psi_3(0),\quad \forall\psi\in X.
\]
It is straightforward to verify that $X_0 = p^{-1}(0,\infty)$ and $\partial X_0 = p^{-1}(0)$. 

 Let $K$ denote the global attractor of $T(t)$ with the decomposition $K=X_\partial\cup A_0$, where $A_0=K\cap X^0\subset X^0$.
By Claim 2, no trajectory emanating from $X^0$ can converge to $E^0$, which implies $\omega(\psi)\subset A_0$ for all $\psi\in X^0$. Since $A_0$ is compact and $p$ is continuous, the minimum $\eta_1=\min_{\phi\in A_0}p(\phi)$ exists and satisfies $\eta_1>0$.
For any $\phi\in\omega(\psi)$, there exists a sequence $\{t_n\}$ with $t_n\to\infty$ such that $T(t_n)\psi\to\phi$. Continuity of $p$ gives $p(\phi)=\lim_{n\to\infty}p(T(t_n)\psi)\geq\eta_1$. Hence every limit point of $\{p(T(t)\psi)\}_{t\geq0}$ is bounded below by $\eta_1$, and therefore
\[
\liminf_{t\to\infty}p\big(T(t)\psi\big)=\liminf_{t\to\infty}I(t,\psi)\geq\eta_1,\quad \forall\psi\in X^0.
\]

We now derive uniform positive lower bounds for $S(t)$, $V(t)$ and $R(t)$. All population components admit a uniform upper bound $\Lambda/d_0$. For $S(t)$, we have
\[
\frac{dS(t)}{dt}
=\Lambda  - \iota I e^{-mI}S - (d_0+\varpi)S + \upsilon V + \vartheta R
\ge \Lambda  - \left(\iota \frac{\Lambda }{d_0}+d_0+\varpi\right)S.
\]
Let $K_S=\iota\Lambda/d_0+d_0+\varpi$. By the comparison principle for scalar differential inequalities,
\[
S(t)\geq\left(S(s_0)-\frac{\Lambda}{K_S}\right)e^{-K_S(t-s_0)}+\frac{\Lambda}{K_S},\quad \forall\,t\geq s_0.
\]
Taking $t\to\infty$, we obtain $\liminf_{t\to\infty}S(t)\geq\Lambda/K_S>0$.

Similarly, the equation for $V(t)$ gives
\[
\frac{dV(t)}{dt} = \varpi S - \sigma_0\iota I e^{-mI}V - (d_0+\upsilon)V
\ge \varpi S - \left(\sigma_0\iota \frac{\Lambda }{d_0}+d_0+\upsilon\right)V.
\]
Since $\liminf_{t\to\infty}S(t)>0$, an analogous comparison argument yields $\liminf_{t\to\infty}V(t)>0$.

For $R(t)$, we estimate
\[
\frac{dR(t)}{dt} = \frac{\gamma_0 I}{1+\alpha I} - (d_0+\vartheta)R
\ge \frac{\gamma_0 \eta_2}{1+\alpha \Lambda/d_0} - (d_0+\vartheta)R.
\]
Solving this linear differential inequality directly implies $\liminf_{t\to\infty}R(t)>0$.

 Therefore,
\[
\liminf_{t\to\infty}\min\big\{S(t,\psi),V(t,\psi),I(t,\psi),R(t,\psi)\big\}\geq\upsilon_0.
\]
which completes the verification of \eqref{4.1}.
\end{proof}

\subsection{Global asymptotic stability of the unique endemic equilibrium $E^*$}

In this section, we prove the global asymptotic stability of the unique endemic equilibrium \(E^{*}\) for system \eqref{1.1} in the ODE case \(\tau=0\) using the geometric method based on the third additive compound matrix, following the rigorous theoretical framework established in \cite{1999Global,LI2000295}. For the ODE case without time delay (\(\tau=0\)), the system is an autonomous finite-dimensional ordinary differential system, so the additive compound matrix technique for ODEs is applicable.

When \(\tau=0\), system \eqref{1.1} reduces to
\begin{equation}\label{4.2}
	\begin{cases}
		\displaystyle \frac{d S}{d t}=\Lambda -\iota I e^{-m I} S-(d_0+\varpi) S+\upsilon V+\vartheta R, \\[4pt]
		\displaystyle \frac{d V}{d t}=\varpi S-\sigma_0 \iota I e^{-m I} V-(d_0+\upsilon) V, \\[4pt]
		\displaystyle \frac{d I}{d t}=\iota I e^{-m I}(S+\sigma_0 V)-(d_0+\gamma_0 +d) I, \\[4pt]
		\displaystyle \frac{d R}{d t}=\frac{\gamma_0 I}{1+\alpha I}-(d_0+\vartheta) R.
	\end{cases}
\end{equation}

From Theorems 2.1--2.4 and Theorem 4.2, we derive the following preliminary results:
\begin{enumerate}
	\item All solutions starting from positive initial data remain nonnegative and are uniformly ultimately bounded;
	\item If \(\mathcal{R}_0 > 1\) and
	\[
	m \geq \frac{\sigma_0 \iota \varpi\left(1-\sigma_0\right)}{\left(d_0+\upsilon+\varpi\right)\left(d_0+\upsilon+\sigma_0 \varpi\right)},
	\]
	then system \eqref{4.2} has a unique strictly positive endemic equilibrium \(E^{*}\). Furthermore, the system is uniformly persistent on the interior of \(\Omega\): there exists a constant \(\upsilon_0  > 0\) such that 
	\[
	\liminf_{t\to\infty}
	\min
	\big\{
	S(t,\psi),\;
	V(t,\psi),\;
	I(t,\psi),\;
	R(t,\psi)
	\big\}
	\geq
	\upsilon_0 .
	\]
\end{enumerate}

We define the compact absorbing set
\[
\mathcal{K}=\left\{(S,V,I,R)\in\mathbb{R}_{+}^{4}\,\Big|\, S,V,I,R\ge \upsilon_0,\; S+V+I+R\le \frac{\Lambda }{d_0}\right\},
\]
and all positive solutions eventually enter \(\mathcal{K}\) and remain inside for all sufficiently large \(t\).

For any interior state point \(u=(S,V,I,R)\in \Omega\), the Jacobian matrix of system \eqref{4.2} evaluated at \(u\) is
\[
J(u)=
\begin{pmatrix}
	-h_{1}-(d_0+\varpi) & \upsilon & h_{2} & \vartheta \\
	\varpi & -\sigma_0 h_{1}-(d_0+\upsilon) & h_{3} & 0 \\
	h_{1} & \sigma_0 h_{1} & -\big(h_{2}+h_{3}\big)-(d_0+\gamma_0 +d) & 0 \\
	0 & 0 & h_{4} & -(d_0+\vartheta)
\end{pmatrix},
\]
where the state-dependent coefficients are defined as
\[
\begin{aligned}
	&h_{1}=\iota I e^{-m I},\quad h_{2}=\iota(m I-1)e^{-m I}S,\\
	&h_{3}=\sigma_0\iota(m I-1)e^{-m I}V,\quad h_{4}=\frac{\gamma_0 }{(1+\alpha I)^{2}}.
\end{aligned}
\]
The third additive compound matrix corresponding to \(J(u)\) takes the form
\[
J^{[3]}(u)=
\begin{pmatrix}
	-(1+\sigma_0)h_{1}-\big(h_{2}+h_{3}\big)-q_{1} & 0 & 0 & \vartheta \\
	h_{4} & -(1+\sigma_0)h_{1}-q_{2} & h_{3} & -h_{2} \\
	0 & \sigma_0 h_{1} & -\big(h_{1}+h_{2}+h_{3}\big)-q_{3} & \upsilon \\
	0 & -h_{1} & \varpi & -\big(\sigma_0 h_{1}+h_{2}+h_{3}\big)-q_{4}
\end{pmatrix},
\]
with constant parameter combinations
\[
\begin{aligned}
	q_{1}&=(d_0+\varpi)+(d_0+\upsilon)+(d_0+\gamma_0 +d),\\
	q_{2}&=(d_0+\varpi)+(d_0+\upsilon)+(d_0+\vartheta),\\
	q_{3}&=(d_0+\varpi)+(d_0+\gamma_0 +d)+(d_0+\vartheta),\\
	q_{4}&=(d_0+\upsilon)+(d_0+\gamma_0 +d)+(d_0+\vartheta).
\end{aligned}
\]

Along an arbitrary positive trajectory \(u(t)=(S(t),V(t),I(t),R(t))\) of system \eqref{4.2}, we derive the time-varying linear compound system, which constitutes the core object of the geometric stability criterion:
\begin{equation}\label{4.3}
	\dot{U}=J^{[3]}\big(u(t)\big)U,\qquad U=(X,Y,Z,W)^T.
\end{equation}

To investigate the global asymptotic stability of system \eqref{4.2}, we construct the Lyapunov function
\[
V(t;u,U) = \max\{V_1, V_2, V_3\},
\]
where
\[
V_1 = |X|, \quad V_2 = |Y|, \quad V_3 =
\begin{cases}
	|Z + W|, & WZ \ge 0, \\
	\max\big\{|Z|,\, |W|\big\}, & WZ < 0.
\end{cases}
\]
 It is straightforward to verify that
\[
|W + Z| \le V_3,\quad \forall\,(W, Z) \in \mathbb{R}^2.
 \]
There exist positive constants \(c_1,c_2\) such that
\[
c_1\bigl(|X|+|Y|+|Z|+|W|\bigr)
\leq
V
\leq
c_2\bigl(|X|+|Y|+|Z|+|W|\bigr).
\]

To ensure the Dini derivative of \(V\) is uniformly negative over the compact absorbing set \(\mathcal{K}\), we impose the following parameter constraints, which hold for every trajectory point \((S(t),V(t),I(t))\):
\[
\begin{cases}
	\vartheta +\iota e^{-m I}(S+\sigma_0 V)< q_{1},\\[4pt]
	\displaystyle \frac{\gamma_0}{(1+\alpha I)^2} + \iota e^{-mI}\big[ (\sigma_0 V+S) (mI+1)  \big]    < q_{2},\\[4pt]
	\iota I e^{-m I}  + \iota e^{-mI}(S+\sigma_0 V) < 2 d_0+\gamma_0 +d+\vartheta,\\[4pt]
	\iota I e^{-m I} + \iota e^{-mI} (S+\sigma_0 V)< q_{4},\\[4pt]
	\iota I e^{-m I}  + \iota e^{-mI} (S+\sigma_0 V)  < q_{3}.
\end{cases}\tag{P}\label{eq:P}
\]

Let \(D_+V\) denote the upper-right Dini derivative of \(V\) along the solutions of the time-varying  compound system \eqref{4.3}. We calculate \(D^+V\) separately for each case in the subsequent derivation.

\begin{enumerate}
\item If $V_1\ge V_2,V_3$, then $V=V_1$.
\begin{align*}
	D_+ V &= D_+ V_1 = D_+ |X| \\
	&\le -\big[(1+\sigma_0)h_1+h_2+h_3+q_1\big]|X|+\vartheta |W| \\
	&\le -\big[(1+\sigma_0)\iota e^{-m I}I+\iota(m I-1)e^{-m I}S+\sigma_0\iota(m I-1)e^{-m I}V\big]|X|
	+\vartheta |W|-q_1|X| \\
	&\le \big[\iota e^{-m I}(S+\sigma_0 V)-q_1\big]|X|+\vartheta |W| \\
	&\le \big[\vartheta +\iota e^{-m I}(S+\sigma_0 V)-q_1\big]|X| \\
	&= \big[\vartheta +\iota e^{-m I}(S+\sigma_0 V)-q_1\big] V.
\end{align*}

\item If $V_2\ge V_1,V_3$, then $V=V_2$.
\begin{align*}
	D_+ V &= D_+ V_2 = D_+ |Y| \\
	&\leq h_4 |X| - \big[(1+\sigma_0) h_1 + q_2\big] |Y| + \big|h_3\operatorname{sign}(Y)Z\big| + \big|-h_2\operatorname{sign}(Y)W\big| \\
	&\leq h_4 |X| - \big[(1+\sigma_0) h_1 + q_2\big] |Y| + |h_3|\,|Z| + |h_2|\,|W| \\
	&= \frac{\gamma_0}{(1+\alpha I)^2} |X| - \big[(1+\sigma_0)\iota I e^{-mI} + q_2\big] |Y| \\
	&\quad + \big|\iota (m I - 1)\sigma_0 e^{-mI} V\big|\,|Z| + \big|-\iota (mI - 1) e^{-mI} S\big|\,|W| \\
	&\leq \left\{ \frac{\gamma_0}{(1+\alpha I)^2} + \big|\iota(mI-1)\sigma_0 e^{-mI}V\big| + \big|\iota(mI-1)e^{-mI}S\big| - \big[(1+\sigma_0)\iota I e^{-mI}+q_2\big] \right\} |Y| \\
	&\leq  \left\{ \frac{\gamma_0}{(1+\alpha I)^2} + \iota e^{-mI}\big[ (\sigma_0 V+S) (mI+1)  \big]  - q_2 \right\} V.
\end{align*}

\item If $WZ\ge0$, we have $|W+Z|=|W|+|Z|$. When $V_3\ge V_1,V_2$, $V=|W+Z|$.
\[
\begin{aligned}
	D_+ V &= D_+ V_3 = D_+ |W+Z| \le \operatorname{sign}(W+Z)(\dot{W}+\dot{Z}) \\
	&= \sigma_0 h_1 \operatorname{sign}(W+Z)Y - \bigl[(h_1+h_2+h_3)+q_3\bigr]|Z|+ \upsilon |W| - h_1 \operatorname{sign}(W+Z)Y + \varpi |Z| \\
	&\quad - \bigl[(\sigma_0 h_1+h_2+h_3)+q_4\bigr]|W| \\
	&= h_1(\sigma_0-1)\operatorname{sign}(W+Z)Y
	- \bigl[(h_1+h_2+h_3)+q_3\bigr]|Z|+ \upsilon |W| + \varpi |Z| \\
	&\quad - \bigl[(\sigma_0 h_1+h_2+h_3)+q_4\bigr]|W| \\
	&\leq h_1(1-\sigma_0)|Y|
	- \bigl[(h_1+h_2+h_3)+q_3\bigr]|Z|+ \upsilon |W| + \varpi |Z| \\
	&\quad - \bigl[(\sigma_0 h_1+h_2+h_3)+q_4\bigr]|W| \\
	&\leq h_1(1-\sigma_0)|Y|+ \upsilon |W| + \varpi |Z|\\
	&\qquad-\Bigl[\iota I e^{-mI} + \iota(mI-1)e^{-mI}(S+\sigma_0 V)+ (d_0+\varpi)+(d_0+\gamma_0+d)+(d_0+\vartheta)\Bigr]|Z| \\
	&\qquad - \Bigl[\sigma_0 \iota I e^{-mI} + \iota(mI-1)e^{-mI}(S+\sigma_0 V)+ (d_0+\upsilon)+(d_0+\gamma_0+d)+(d_0+\vartheta)\Bigr]|W| \\
	&\leq\left[ \iota I e^{-m I}  + \iota e^{-mI}(S+\sigma_0 V) 
	-(d_0+\gamma_0+d)-(d_0+\vartheta) \right]V.
\end{aligned}
\]

\item If $WZ<0$ and $|Z|\le |W|$, then $V_3\ge V_1,V_2$ yields $V=|W|$.
\[
\begin{aligned}
	D_+ V &= D_+ V_3 = D_+ |W| \le \operatorname{sign}(W)\dot{W} \\
	&= - h_1 \operatorname{sign}(W)Y + \varpi\operatorname{sign}(W)Z - \left[\sigma_0 h_1 + h_2 + h_3 + q_4\right] |W| \\
	&= - h_1 \operatorname{sign}(W)Y - \varpi |Z| - \left[\sigma_0 h_1 + h_2 + h_3 + q_4\right] |W| \\
	&\leq h_1 |Y| - \varpi |Z| - \left[\sigma_0 \iota I e^{-mI} + \iota (mI-1) e^{-mI} (S+\sigma_0 V) +q_4\right] |W| \\
	&\leq h_1 |Y| - \varpi |Z| + \iota e^{-mI}(S+\sigma_0 V)|W| - q_4 |W| \\
	&\leq \left[ h_1 + \iota e^{-mI} (S+\sigma_0 V) -  q_4 \right]|W|\\
	&= \left[ \iota I e^{-m I} + \iota e^{-mI} (S+\sigma_0 V) -  q_4 \right]V.
\end{aligned}
\]

\item If $WZ<0$ and $|W|<|Z|$, then $V_3\ge V_1,V_2$ yields $V=|Z|$.
\[
\begin{aligned}
	D_+ V &= D_+ V_3 = D_+ |Z| \le \operatorname{sign}(Z)\dot{Z} \\
	&= \sigma_0 h_1 \operatorname{sign}(Z)Y - \left[h_1 + h_2 + h_3 + q_3\right] |Z| + \upsilon \operatorname{sign}(Z)W \\
	&= \sigma_0 h_1 \operatorname{sign}(Z)Y - \left[h_1 + h_2 + h_3 + q_3\right] |Z| - \upsilon |W| \\
	&\leq \sigma_0 h_1 |Y| - \left[\iota I e^{-mI} + \iota (mI-1) e^{-mI} (S+\sigma_0 V) +q_3\right] |Z| - \upsilon |W| \\
	&\leq \sigma_0 h_1 |Z| + \iota e^{-mI}(S+\sigma_0 V)|Z| - \upsilon |W| - q_3 |Z| \\
	&\leq \left[ \sigma_0 h_1 + \iota e^{-mI} (S+\sigma_0 V) - q_3 \right]|Z|\\
	&\leq  \left[ \iota I e^{-m I}  + \iota e^{-mI} (S+\sigma_0 V) - q_3 \right]V.
\end{aligned}
\]

\end{enumerate}

	Define
	\[\begin{aligned}
		k_1&=\inf_{\mathcal{K}}\big\{\vartheta +\iota \mathrm{e}^{-m I}(S+\sigma_0 V)-q_1\big\},\\
		k_2&=\inf_{\mathcal{K}}\Bigg\{\frac{\gamma_0}{(1+\alpha I)^2} + \iota e^{-mI}\big[ (\sigma_0 V+S) (mI+1)  \big]   -q_2\Bigg\},\\
		k_3&=\inf_{\mathcal{K}}\big\{\iota I e^{-m I}  + \iota e^{-mI}(S+\sigma_0 V)-(2 d_0+\gamma_0 +d+\vartheta)\big\},\\
		k_4&=\inf_{\mathcal{K}}\big\{\iota I e^{-m I} + \iota e^{-mI} (S+\sigma_0 V)-q_4\big\},\\
		k_5&=\inf_{\mathcal{K}}\big\{\iota I e^{-m I}  + \iota e^{-mI} (S+\sigma_0 V) -q_3\big\},
	\end{aligned}\]
	and set $k=\min\{k_1,k_2,k_3,k_4,k_5\}$. Condition \eqref{eq:P} guarantees $k>0$, such that
	\[
	D_{+}V(t;u,U)\le -kV(t;u,U).
	\]
Therefore, all hypotheses of \cite[Corollary 3.2]{1999Global} hold with $b=k$.
Combining the above volume contraction property with the uniform persistence established earlier, we obtain the global asymptotic stability of the endemic equilibrium $E^*$.

\begin{thm}
	Suppose that \(\mathcal{R}_{0}>1\),
	\[
	m \geq \frac{\sigma_0 \iota \varpi\left(1-\sigma_0\right)}{\left(d_0+\upsilon+\varpi\right)\left(d_0+\upsilon+\sigma_0 \varpi\right)},
	\]
	and condition \eqref{eq:P} holds. Then the unique endemic equilibrium \(E^{*}\) of system \eqref{1.1} is globally asymptotically stable for all initial values \(u_0=(S_0,V_0,I_0,R_0)\in\mathbb{R}_{+}^{4}\) with \(I_0>0\).
\end{thm}

\section{Local Hopf Bifurcation}
In this section, we will consider the case of $\tau\neq 0$ and analyze the Hopf bifurcation phenomenon that occurs at the positive equilibrium $E^*$.
\subsection{Existence of Hopf Bifurcation}
For $\tau\neq 0$, the linearized system of \eqref{1.1} at the equilibrium point $E^*$ leads to the following characteristic equation:
\begin{equation}\label{5.1}
	\lambda^4+a_{11}\lambda^3
	+a_{12}\lambda^2+a_{13}\lambda+a_{14}+(a_{21}\lambda^3+a_{22}\lambda^2+a_{23}\lambda+a_{24})e^{-\lambda\tau}=0.
\end{equation}
Assume that Eq.~\eqref{5.1} admits a purely imaginary root $ \lambda = iw, w>0.$ We take it into the equation and further organize to get

\begin{equation}\label{5.2} \begin{cases}
		w^4-a_{12}w^2+a_{14}=(a_{21}w^3-a_{23}w)\sin w\tau+(a_{22}w^2-a_{24})\cos w\tau,\\
		-a_{11}w^3+a_{13}w=(a_{21}w^3-a_{23}w)\cos w\tau-(a_{22}w^2-a_{24})\sin w\tau.
	\end{cases} 
\end{equation}

Then, we can derive
\begin{equation}\label{5.3}
	w^8+s_1w^6+s_2w^4+s_3w^2+s_4=0, \end{equation}
where

$$ s_1=a_{11}^2-2a_{12}-a_{21}^2,\qquad s_2=a_{12}^2+2a_{14}-2a_{11}a_{13}-a_{22}^2+2a_{21}a_{23}, $$

$$ s_3=a_{13}^2-2a_{12}a_{14}+2a_{22}a_{24}-a_{23}^2, \qquad s_4=a_{14}^2-a_{24}^2. $$

Letting $ m=w^2 $, we have
\begin{equation}\label{5.4}
	m^4+s_1m^3+s_2m^2+s_3m+s_4=0. 
\end{equation}
We further investigate the root distribution of equation \eqref{5.4} by adopting the analytical approach proposed by Yan and Li in \cite{Xiang-Ping}.

\begin{lem}There is at least one positive root in Eq.~\eqref{5.4} for $s_{4}<0$.
\end{lem}
\begin{proof} Denote
	\begin{equation}\label{5.5}
		H(m)=m^4+s_1m^3+s_2m^2+s_3m+s_4.
	\end{equation}
	Since $H(0)=s_4<0$ and $\lim\limits_{m \to +\infty} H(m)=+\infty$, Eq.~\eqref{5.5} has at least one $m_{0}>0$ such that $H(m_{0})=0$.
\end{proof}
Next, when $s_{4}\ge 0$, denote
$$
{r_1}^*=\frac{s_{2}}{2}-\frac{3s_{1}^{2}}{16},\quad
{r_2}^*=\frac{s_{1}^{3}}{32}-\frac{s_{1}s_{2}}{8}+\frac{s_{3}}{4},\quad
{r_3}^*=\left(\frac{{r_2}^*}{2}\right)^{2}+\left(\frac{{r_1}^*}{3}\right)^{3}.
$$
Then, we define
$$\begin{aligned}
	m_1&=-\frac{s_{1}}{4}+\sqrt[3]{-\frac{{r_2}^*}{2}+\sqrt{{r_3}^*}}+\sqrt[3]{-\frac{{r_2}^*}{2}-\sqrt{{r_3}^*}}, &{r_3}^*&>0,\\
	m_2&=\max\left\{-\frac{s_{1}}{4}-2\sqrt[3]{\frac{{r_2}^*}{2}},-\frac{s_1}{4}+2\sqrt[3]{\frac{{r_2}^*}{2}}\right\}, &{r_3}^*&=0,\\
	m_3&=\max\left\{-\frac{s_1}{4}+2\operatorname{Re}\{\delta\},-\frac{s_1}{4}+2\operatorname{Re}\{\delta\varepsilon\},-\frac{s_1}{4}+2\operatorname{Re}\{\delta\bar{\varepsilon}\}\right\}, &{r_3}^*&<0,
\end{aligned}$$
where $\delta$ is one of cube roots of the complex number $-\dfrac{{r_2}^*}{2}+\sqrt{{r_3}^*}$ and $\varepsilon=\dfrac{-1+\sqrt{3}i}{2}$. Therefore, we have the following:
\begin{lem}
	If $s_4\geqslant 0$, Eq.~\eqref{5.5} exists at least one positive root if one of the following conditions holds:
	(a) ${r_3}^*>0$, $m_1>0$ and $H(m_1)<0$,
	(b) ${r_3}^*=0$, $m_2>0$ and $H(m_2)<0$,
	(c) ${r_3}^*<0$, $m_3>0$ and $H(m_3)<0$.
	If $s_4\geqslant 0$, there is no positive root for Eq.~\eqref{5.5} if one of the following conditions holds:
	(a) ${r_3}^*>0$, $m_{1}<0$,
	(b) ${r_3}^*=0$, $m_2<0$,
	(c) ${r_3}^*<0$, $m_3<0$.
\end{lem}

Without loss of generality, we suppose that equation \eqref{5.5} has $r$ positive real roots satisfying $1\leqslant r\leqslant 4$, ordered as $m_{1}^*<m_{2}^*<\cdots<m_{r}^*$. Correspondingly, equation \eqref{5.3} admits $r$ positive real roots $w_{k}=\sqrt{m_{k}^*}$ for $1\leqslant k\leqslant r$. Substituting $\lambda=iw$ into system \eqref{5.2} and eliminating trigonometric terms yields
\begin{equation}\label{5.6}
	\sin w\tau=\frac{(w^{4}-a_{12}w^{2}+a_{14})(a_{21}w^{3}-a_{23}w)-(-a_{11}w^{3}+a_{13}w)(a_{22}w^{2}-a_{24})}{(a_{21}w^{3}-a_{23}w)^{2}+(a_{22}w^{2}-a_{24})^{2}}\triangleq \zeta(w).
\end{equation}

Solving the above equation for $\tau$, we obtain a family of critical delays
\begin{equation}\label{5.7}
	\tau_n^{(k)}
	=
	\frac{1}{w_k}\arcsin\zeta(w_k)
	+
	\frac{2n\pi}{w_k},
	\qquad
	k=1,2,\dots,r,\quad n=0,1,2,\dots.
\end{equation}

The pair $\pm iw_k$ are purely imaginary roots of characteristic equation \eqref{5.1}. For each fixed index $k$, the sequence $\{\tau_n^{(k)}\}_{n\ge0}$ is strictly increasing in $n$ with $\lim\limits_{n\to+\infty}\tau_n^{(k)}=+\infty$.
Among all critical delays $\{\tau_n^{(k)}\}$, there uniquely exist indices $k_0\in\{1,2,\dots,r\}$ and $n_0\in\{0,1,2,\dots\}$ such that
\[
\tau_{n_0}^{(k_0)}=\min\big\{\tau_n^{(k)}\;\big|\;k=1,2,\dots,r,\;n=0,1,2,\dots\big\}.
\]

Define the minimal critical delay, the corresponding frequency and the associated squared‑frequency by
\begin{equation}\label{5.8}
	\tau_0 = \tau_{n_0}^{(k_0)},\quad w_0 = w_{k_0},\quad m_0 = m_{k_0}^*.
\end{equation}
Here $\lambda=iw_0$ is the purely imaginary root associated with the first loss of stability at $\tau=\tau_0$.

Let \(\lambda(\tau) = \xi(\tau) + i\omega(\tau)\) be the characteristic root satisfying \(\xi(\tau_n^{(k)}) = 0,\ \omega(\tau_n^{(k)}) = w_k\) and \(m_k^* = w_k^2\). Differentiating both sides of Eq.~\eqref{5.1} with respect to \(\tau\), we derive
\[
\left( \frac{\mathrm{d}\lambda}{\mathrm{d}\tau} \right)^{-1} = -\frac{4\lambda^3 + 3a_{11} \lambda^2 + 2a_{12} \lambda + a_{13}}{\lambda \left( \lambda^4 + a_{11} \lambda^3 + a_{12} \lambda^2 + a_{13} \lambda + a_{14} \right)} + \frac{3a_{21} \lambda^2 + 2a_{22} \lambda + a_{23}}{\lambda \left( a_{21} \lambda^3 + a_{22} \lambda^2 + a_{23} \lambda + a_{24} \right)} - \frac{\tau}{\lambda}.
\]

Substitute \(\lambda=iw_k\) into the above expression and take its real part, then we obtain
\begin{align*}
	\operatorname{sign} \left\{ \operatorname{Re} \left( \frac{\mathrm{d}\lambda}{\mathrm{d}\tau} \right)^{-1} \bigg|_{\lambda = i w_k} \right\}
	&= \operatorname{sign} \left\{ \frac{4w_k^6 + w_k^4 (3a_{11}^2 - 6a_{12}) + w_k^2 (4a_{14} + 2a_{12}^2 - 4a_{13} a_{11}) + (a_{13}^2 - 2a_{12} a_{14})}{w_k^2 (a_{11} w_k^2 - a_{13})^2 + (w_k^4 - a_{12} w_k^2 + a_{14})^2} \right. \\
	&\quad \left. + \frac{-3a_{21}^2 w_k^4 + w_k^2 (4a_{21} a_{23} - 2a_{22}^2) + (2a_{24} a_{22} - a_{23}^2)}{w_k^2 (a_{21} w_k^2 - a_{23})^2 + (a_{22} w_k^2 - a_{24})^2} \right\}.
\end{align*}

From the identity obtained by squaring both sides and adding the two equations in \eqref{5.2},  we have
\[
w_k^{2}(a_{11}w_k^{2}-a_{13})^{2}+(w_k^{4}-a_{12}w_k^{2}+a_{14})^{2}=w_k^{2}(a_{21}w_k^{2}-a_{23})^{2}+(a_{22}w_k^{2}-a_{24})^{2}.
\]
The denominator on the right-hand side is a sum of squares and thus strictly positive. Combining the relation \(m_k^*=w_k^2\), we arrive at
\[
\operatorname{sign}\left\{\frac{\mathrm{d}\operatorname{Re}\lambda}{\mathrm{d}\tau}\bigg|_{\tau=\tau_{n}^{(k)}}\right\}=\operatorname{sign}\left\{\operatorname{Re}\left(\frac{\mathrm{d}\lambda}{\mathrm{d}\tau}\right)^{-1}\bigg|_{\lambda=iw_k}\right\}=\operatorname{sign}\frac{H'(m_k^*)}{w_k^{2}\left(a_{21}w_{k}^{2}-a_{23}\right)^{2}+\left(a_{22}w_{k}^{2}-a_{24}\right)^{2}}.
\]
Since the denominator is always positive, the sign of \(\dfrac{\mathrm{d}\operatorname{Re}\lambda}{\mathrm{d}\tau}\) evaluated at \(\tau=\tau_n^{(k)}\) coincides exactly with the sign of \(H'(m_k^*)\). Combining the previous arguments, we arrive at the following theorem.

\begin{thm}
	Let \(\tau_n^{(k)}\), $w_0$, and \(\tau_0\) be defined by
	\eqref{5.7} and \eqref{5.8}.
	
	\begin{enumerate}
		\item[{\rm (i)}]
		If Eq.~\eqref{5.5} admits no positive solution, then the equilibrium $E^*$ remains locally asymptotically stable for all $\tau>0$.
		
		\item[{\rm (ii)}]
		If Eq.~\eqref{5.5} possesses at least one positive solution, then $E^*$ is locally asymptotically stable whenever
		$
		0\leq\tau<\tau_0,
		$
		whereas it becomes unstable for
		$
		\tau>\tau_0.
		$
		Furthermore, if
		$
		H'(m_k^*)>0,
		$
		then system \eqref{1.1} undergoes a Hopf bifurcation at the equilibrium $E^*$ as $\tau$ crosses the critical value $\tau_n^{(k)}$.
	\end{enumerate}
\end{thm}

\subsection{Bifurcation Direction and Stability Characteristics}

To characterize the local properties of periodic solutions bifurcating from
\(E^*\), we utilize the center manifold reduction and normal form theory developed in \cite{hassard1981theory}. The subsequent analysis follows the framework proposed in \cite{balasubramaniam2015stability,xu2022delayed}, with suitable adaptations for the present model.

Let \(\tau=\tau_0+\varepsilon\), where \(\varepsilon\) denotes a small perturbation from the critical delay \(\tau_0\). When \(\varepsilon=0\), system \eqref{1.1} undergoes a Hopf bifurcation at the equilibrium \(E^*\). Define the perturbation variables by
\[
u_1(t)=S(t)-S^*,\quad
u_2(t)=V(t)-V^*,\quad
u_3(t)=I(t)-I^*,\quad
u_4(t)=R(t)-R^*.
\]

Accordingly, system \eqref{1.1} can be rewritten as the following functional differential equation
\begin{equation}\label{5.9}
	\dot{u}(t) = L_{\varepsilon}(u_{t}) + F(\varepsilon, u_{t}),
\end{equation}
where \(u(t) = (u_{1}(t), u_{2}(t), u_{3}(t), u_{4}(t))^{T} \in \mathbb{R}^{4}\), and the history segment \(u_{t}\in C([-\tau,\ 0], \mathbb{R}^{4})\) satisfies \(u_{t}(q) = u(t + q)\) for all \(q \in [-\tau,\ 0]\). Here, \(L_{\varepsilon}: C([-\tau,\ 0], \mathbb{R}^{4}) \to \mathbb{R}^{4}\) and \(F: \mathbb{R} \times C([-\tau,\ 0], \mathbb{R}^{4}) \to \mathbb{R}^{4}\).  For any \(\phi \in C([-\tau,\ 0], \mathbb{R}^{4})\), we define
\[
L_{\varepsilon}(\phi) = A \phi(0) + B \phi(-\tau),
\]
in which
\[
A = \begin{pmatrix}
	-(d_0+\varpi) & \upsilon & 0 & \vartheta \\
	\varpi & -(d_0+\upsilon) &0 & 0 \\
	0 &0 & -(d_0+\gamma_0 +d) &0\\
	0 & 0 & \dfrac{\gamma_0 }{(1+\alpha I^*)^2} & -(d_0+\vartheta)
\end{pmatrix},
\]
\[
B = \begin{pmatrix}
	-\iota I^*e^{-mI^*} & 0 & -\iota S^* e^{-mI^*}+\iota I^* S^* m e^{-mI^*} & 0 \\
	0 & -\sigma_0\iota I^*e^{-mI^*} &-\sigma_0\iota V^* e^{-mI^*}+\sigma_0\iota I^* V^* m e^{-mI^*} & 0 \\
	\iota I^*e^{-mI^*} &\sigma_0\iota I^*e^{-mI^*} & \iota(S^*+\sigma_0 V^*) e^{-mI^*}-m\iota(S^*+\sigma_0 V^*) e^{-mI^*}I^* &0\\
	0 & 0 & 0&0
\end{pmatrix},
\]
and the nonlinear term is expressed as
\[
F(\varepsilon, \phi) =
\left(
\begin{array}{c}
	-\iota\phi_3(0)e^{-m\phi_3(-\tau)}\phi_1(0)\\
	-\sigma_0\iota\phi_3(0)e^{-m\phi_3(-\tau)}\phi_2(0) \\
	\iota\phi_3(0)e^{-m\phi_3(-\tau)}\big(\phi_1(0)+\sigma_0\phi_2(0)\big)\\
	-\dfrac{\gamma_0  \alpha}{(1+\alpha I^*)^3} \phi_3^2(0)
\end{array}
\right).
\]

By virtue of the Riesz representation theorem for bounded linear operators on $C([-\tau,0],\mathbb{C}^4)$, the linear operator \(L_{\varepsilon}\) admits the integral representation
\[
L_{\varepsilon}\phi
=
\int_{-\tau}^{0}
\mathrm{d}\eta(q,\varepsilon)\,\phi(q),
\]
where \(\eta(q,\varepsilon)\) denotes a matrix-valued function of bounded variation defined on \([-\tau,0]\), and \(\mathrm{d}\eta(q,\varepsilon)\) stands for the corresponding Lebesgue–Stieltjes measure.
We further set
\[
\mathrm{d}\eta(q, \varepsilon) = \left(A \delta(q) + B \delta(q + \tau)\right)\mathrm{d} q,
\]
in which $\delta(\cdot)$ denotes the Dirac delta distribution satisfying
\[
\int_{-\tau}^0\delta(q)\phi(q)\mathrm{d}q=\phi(0),\qquad
\int_{-\tau}^0\delta(q+\tau)\phi(q)\mathrm{d}q=\phi(-\tau)
\]
for any continuous function $\phi$.

From now on, all vector-valued functions are considered over \(\mathbb{C}^4\) instead of \(\mathbb{R}^4\) to facilitate complex eigenvalue analysis.
For any $\phi\in C^{1}([-\tau,0],\mathbb{C}^4)$, we define
\[
\mathcal{A}(\varepsilon)\phi(q) = 
\begin{cases}
	\dfrac{\mathrm{d}\phi(q)}{\mathrm{d}q}, & q \in [-\tau, 0), \\[6pt]
	\displaystyle\int_{-\tau}^{0} \mathrm{d}\eta(q, \varepsilon) \phi(q) = L_{\varepsilon}\phi, & q = 0,
\end{cases}
\]
and
\[
N(\varepsilon)\phi(q) = 
\begin{cases}
	0, & q \in [-\tau, 0), \\
	F(\varepsilon, \phi), & q = 0.
\end{cases}
\]

Accordingly, system \eqref{5.9} can be rewritten in the equivalent abstract form
\begin{equation}\label{5.10}
	\dot{u}_{t} = \mathcal{A}(\varepsilon)u_{t} + N(\varepsilon)u_{t}.
\end{equation}

For any $\zeta \in C^{1}([0, \tau], \mathbb{C}^{4})$, the adjoint operator $\mathcal{A}^{*}(\varepsilon)$ is defined by
\[
\mathcal{A}^{*}(\varepsilon)\zeta(r) = 
\begin{cases}
	-\dfrac{\mathrm{d}\zeta(r)}{\mathrm{d}r}, & r \in (0, \tau], \\[6pt]
	\displaystyle\int_{-\tau}^{0} \mathrm{d}\eta^{T}(q, \varepsilon)\zeta(-q), & r = 0,
\end{cases}
\]
together with the bilinear inner product
\begin{equation}\label{5.11}
	\langle \zeta, \phi \rangle
	= \overline{\zeta}^{T}(0) \phi(0)
	- \int_{-\tau}^{0} \int_{0}^{q}
	\overline{\zeta}^{T}(\beta-q) \mathrm{d}\eta(q) \phi(\beta)\mathrm{d}\beta, 
\end{equation}
where $\eta(q)=\eta(q,0)$. The operators $\mathcal{A}(\varepsilon)$ and $\mathcal{A}^{*}(\varepsilon)$ are adjoint with respect to the inner product $\langle\cdot,\cdot\rangle$.

It is well known that \(\pm i\omega_0\) are eigenvalues of the operator \(\mathcal{A}(0)\). Then \(\mp i\omega_0\) belong to the spectrum of its adjoint operator \(\mathcal{A}^{*}(0)\). We then seek the corresponding eigenvectors in the form
\[
\rho(q)
=
(1,\rho_2,\rho_3,\rho_4)^{T}e^{i\omega_0 q},
\]
and
\[
\rho^{*}(r)
=
\overline{G}\,(1,\rho_2^{*},\rho_3^{*},\rho_4^{*})^{T}
e^{i\omega_0 r},
\]
which correspond to the eigenvalue \(i\omega_0\) of \(\mathcal{A}(0)\) and the eigenvalue \(-i\omega_0\) of \(\mathcal{A}^{*}(0)\), respectively.

It follows that
\[
\big(A + B e^{-i \omega_{0} \tau_{0}} - i \omega_{0} I_4\big) \rho(0) = 0,
\]
namely
\[
\begin{cases}
	-\bigl(d_0 + \varpi + i\omega_0\bigr) + \upsilon \rho_2 + \vartheta \rho_4
	+ \Bigl[-\iota I^* e^{-mI^*} + \bigl(-\iota S^* e^{-mI^*} + \iota I^* S^* m e^{-mI^*}\bigr)\rho_3\Bigr]e^{-i\omega_0\tau_0} = 0,\\[4pt]
	\varpi - \bigl(d_0 + \upsilon + i\omega_0\bigr)\rho_2
	+ \Bigl[-\sigma_0 \iota I^* e^{-mI^*}\rho_2 + \bigl(-\sigma_0 \iota V^* e^{-mI^*} + \sigma_0 \iota I^* V^* m e^{-mI^*}\bigr)\rho_3\Bigr]e^{-i\omega_0\tau_0}=0,\\[4pt]
	-\bigl(d_0+\gamma_0+d+i\omega_0\bigr)\rho_3
	+\Bigl[\iota I^* e^{-mI^*}+\sigma_0\iota I^* e^{-mI^*}\rho_2
	+\bigl(\iota(S^*+\sigma_0V^*)e^{-mI^*}-m\iota(S^*+\sigma_0V^*)e^{-mI^*}I^*\bigr)\rho_3\Bigr]e^{-i\omega_0\tau_0}=0,\\[4pt]
	\dfrac{\gamma_0}{(1+\alpha I^*)^2}\rho_3 - \bigl(d_0+\vartheta+i\omega_0\bigr)\rho_4 = 0.
\end{cases}
\]

Further computation yields
$$
\begin{aligned}
	\rho_2=&\frac{1}{\upsilon}\left\{q_1+\left[\iota S^* e^{-mI^*}e^{-i\omega_0\tau_0}-\iota I^* S^* m e^{-mI^*}e^{-i\omega_0\tau_0}-\frac{\gamma_0 \vartheta}{(d_0+\vartheta+i\omega_0)(1+\alpha I^*)^2}\right]\rho_3\right\},\\[4pt]
	\rho_3=&\frac{q_1 q_2 -\upsilon \varpi}{\left(\sigma_0\iota I^* V^* m e^{-mI^*}-\sigma_0\iota V^* e^{-mI^*}\right)\upsilon e^{-i\omega_0\tau_0}-q_2\left[\iota S^* e^{-mI^*}e^{-i\omega_0\tau_0}-\iota I^* S^* m e^{-mI^*}e^{-i\omega_0\tau_0}-\frac{\gamma_0 \vartheta}{(d_0+\vartheta+i\omega_0)(1+\alpha I^*)^2}\right]},\\[4pt]
	\rho_4=&\frac{\gamma_0 }{(d_0+\vartheta+i\omega_0)(1+\alpha I^*)^2}\rho_3,
\end{aligned}
$$
where
$$q_1= d_0+\varpi+i\omega_0+\iota I^*e^{-mI^*}e^{-i\omega_0\tau_0},\qquad q_2= d_0+\upsilon+i\omega_0+\sigma_0\iota I^* e^{-mI^*}e^{-i\omega_0\tau_0}.$$

By adopting the same derivation procedure for \(\rho^{*}(r)\), we obtain the relation
\[
\big(A^{T}+B^{T}e^{i\omega_0\tau_0}+i\omega_0 I_4\big)\rho^{*}(0)=0.
\]
Accordingly, we deduce
\[
\begin{cases}
	-(d_0+\varpi-i\omega_0) + \varpi\rho_2^* 
	+ \bigl[-\iota I^* e^{-m I^*}+ \iota I^* e^{-m I^*} \rho_3^*\bigr] e^{i\omega_0 \tau_0} = 0,\\[4pt]
	\upsilon-(d_0+\upsilon-i\omega_0)\rho_2^*
	+\bigl[-\sigma_0\iota I^* e^{-m I^*}\rho_2^*+ \sigma_0\iota I^* e^{-m I^*} \rho_3^*\bigr] e^{i\omega_0 \tau_0}=0,\\[4pt]
	-(d_0+\gamma_0 +d-i\omega_0)\rho_3^*+\dfrac{\gamma_0 }{(1+\alpha I^*)^2}\rho_4^*
	+\bigl[-\iota S^* e^{-m I^*}+ \iota I^*S^*m e^{-m I^*}\bigr] e^{i\omega_0 \tau_0} \\
	\quad+\bigl[-\sigma_0\iota V^* e^{-m I^*}+ \sigma_0\iota I^*V^*m e^{-m I^*}\bigr] e^{i\omega_0 \tau_0}\rho_2^*
	+\bigl[\iota(S^*+\sigma_0 V^*)e^{-mI^*}-m\iota(S^*+\sigma_0 V^*)e^{-mI^*}I^*\bigr]e^{i\omega_0\tau_0}\rho_3^*=0,\\[4pt]
	\vartheta-(d_0+\vartheta-i\omega_0)\rho_4^*=0.
\end{cases}
\]
and further derive
$$
\begin{aligned}
	\rho_2^*=&\frac{1}{\varpi}\left[( d_0+\varpi-i\omega_0)+\iota I^* e^{-mI^*}e^{i\omega_0\tau_0}-\iota I^*e^{-mI^*}e^{i\omega_0\tau_0}\rho_3^*\right],\\[4pt]
	\rho_3^*=&\frac{( d_0+\upsilon-i\omega_0+\sigma_0\iota I^* e^{-mI^*}e^{i\omega_0\tau_0})( d_0+\varpi-i\omega_0+\iota I^* e^{-mI^*}e^{i\omega_0\tau_0})-\upsilon \varpi}{\sigma_0 \iota I^* e^{-mI^*}e^{i\omega_0\tau_0}\varpi+( d_0+\upsilon-i\omega_0+\sigma_0\iota I^* e^{-mI^*}e^{i\omega_0\tau_0})\iota I^* e^{-mI^*}e^{i\omega_0\tau_0}},\\[4pt]
	\rho_4^*=&\frac{\vartheta}{ d_0+\vartheta-i\omega_0}.
\end{aligned}
$$

Combining formula \eqref{5.11}, we impose the normalization condition $\langle\rho^{*},\rho\rangle=1$. The orthogonality identity $\langle\rho^{*},\overline{\rho}\rangle=0$ holds automatically due to the spectral property of eigenvalues. We can solve for the constant \(\overline{G}\) as follows
\[
\begin{aligned}
	\langle\rho^{*},\rho\rangle 
	&= \overline{G}\left(1+\rho_{2}\overline{\rho_2^*}+\rho_{3}\overline{\rho_3^*}+\rho_{4}\overline{\rho_4^*}\right) \\
	&\quad -\int_{q=-\tau_{0}}^{0}\int_{\beta=0}^{q}\overline{G}\left(1,\overline{\rho_2^*},\overline{\rho_3^*},\overline{\rho_4^*}\right)
	e^{-iw_{0}(\beta-q)}\mathrm{d}\eta(q)(1,\rho_{2},\rho_{3},\rho_{4})^{T}e^{i\omega_{0}\beta}\mathrm{d}\beta \\[4pt]
	&= \overline{G}\Biggl\{1+\sum_{k=2}^{4}\rho_{k}\overline{\rho_k^*}
	+\tau_{0}e^{-i\omega_{0}\tau_0}\Bigl[
	-\iota I^* e^{-mI^*}+\iota I^* e^{-mI^*}\rho_3^* \\
	&\quad +\left(\sigma_0\iota I^* e^{-mI^*}\rho_3^*-\sigma_0\iota I^* e^{-mI^*}\rho_2^*\right)\rho_2 \\
	&\quad +\left(-\iota S^*e^{-mI^*}+\iota I^* S^* me^{-mI^*}\right)\rho_3 \\
	&\quad +\left(-\sigma_0\iota V^* e^{-mI^*}+\sigma_0\iota I^* V^* me^{-mI^*}\right)\rho_2^*\rho_3 \\
	&\quad +\bigl[\iota(S^*+\sigma_0 V^*)e^{-mI^*}-m\iota(S^*+\sigma_0 V^*)e^{-mI^*}I^*\bigr]\rho_3^*\rho_3
	\Bigr]\Biggr\} = 1,
\end{aligned}
\]
which yields
$$\begin{aligned}
	\overline{G}=&\Biggl\{1+\sum_{k=2}^{4}\rho_{k}\overline{\rho_k^*}
	+\tau_{0}e^{-i\omega_{0}\tau_0}\Bigl[
	-\iota I^* e^{-mI^*}+\iota I^* e^{-mI^*}\rho_3^* \\
	&\quad +\left(\sigma_0\iota I^* e^{-mI^*}\rho_3^*-\sigma_0\iota I^* e^{-mI^*}\rho_2^*\right)\rho_2 \\
	&\quad +\left(-\iota S^*e^{-mI^*}+\iota I^* S^* me^{-mI^*}\right)\rho_3 \\
	&\quad +\left(-\sigma_0\iota V^* e^{-mI^*}+\sigma_0\iota I^* V^* me^{-mI^*}\right)\rho_2^*\rho_3 \\
	&\quad +\bigl[\iota(S^*+\sigma_0 V^*)e^{-mI^*}-m\iota(S^*+\sigma_0 V^*)e^{-mI^*}I^*\bigr]\rho_3^*\rho_3
	\Bigr]\Biggr\}^{-1}.
\end{aligned}$$

After that, we introduce the following notations. Under the condition $\varepsilon=0$, we regard $u_{t}$ as the solution of equation \eqref{5.9}. To further characterize the central manifold $C_{0}$, we define the relevant coordinate expressions as given below
\begin{equation}\label{5.12}
	z(t)=\langle\rho^{*},u_{t}\rangle,\qquad W(t,q)=u_{t}-2\operatorname{Re}\{z(t)\rho(q)\}.
\end{equation}

Restricted on the central manifold $C_0$, we have $W(t,q)=W(z(t),\overline{z}(t),q)$, and its formal expansion reads
\begin{equation}\label{5.13}
	W(z(t),\overline{z}(t),q)=W_{20}(q)\frac{z^{2}}{2}+W_{11}(q)z\overline{z}+W_{02}(q)\frac{\overline{z}^{2}}{2}+\cdots.
\end{equation}

Given that $z(t)$ and $\overline{z}(t)$ serve as local coordinates of the central manifold $C_0$ along the directions corresponding to $\rho^{*}$ and $\overline{\rho^{*}}$, we may further deduce that 
\begin{align*}
	\dot{z}(t)&=\langle\rho^{*},\dot{u}_{t}\rangle=\langle\rho^{*},A(0)u_{t}+N(0)u_{t}\rangle,\\
	&=iw_{0}z(t)+\overline{\rho^{*}}^{T}(0)\cdot F\big(0,W(z(t),\overline{z}(t),0)+2\operatorname{Re}\{z(t)\rho(0)\}\big),\\
	&\triangleq iw_{0}z(t)+\overline{\rho^{*}}^{T}(0)\cdot f_{0}(z(t),\overline{z}(t)),
\end{align*}
where $f_{0}(z,\overline{z})=F\big(0,W(z,\overline{z},q)+z(t)\rho(q)+\overline{z}(t)\overline{\rho}(q)\big)$. Notice that
$$
f_{0}=f_{20}\frac{z^{2}}{2}+f_{11}z\overline{z}+f_{02}\frac{\overline{z}^{2}}{2}+f_{21}\frac{z^{2}\overline{z}}{2}+\cdots,
$$
then we rewrite the dynamic equation as
$$
\dot{z}(t)=iw_{0}z(t)+g(z(t),\overline{z}(t)),
$$
with
\begin{equation}\label{5.14}
	g(z(t),\bar{z}(t))=\overline{\rho^{*}}^{T}(0)\cdot f_{0}(z(t),\bar{z}(t))=g_{20}\frac{z^{2}}{2}+g_{11}z\bar{z}+g_{02}\frac{\bar{z}^{2}}{2}+g_{21}\frac{z^{2}\bar{z}}{2}+\cdots.
\end{equation}

From relations \eqref{5.12} and \eqref{5.13}, one can obtain
$$
\begin{aligned}
	u_{t}&=W(t,q)+2\operatorname{Re}\{z(t)\rho(q)\}\\
	&=W_{20}(q)\frac{z^{2}}{2}+W_{11}(q)z\bar{z}+W_{02}(q)\frac{\bar{z}^{2}}{2}+z(t)\rho+\bar{z}(t)\bar{\rho}+\cdots,
\end{aligned}
$$
Substituting this formula into the expression of $F(\varepsilon,\zeta)$ yields
\begin{equation}\label{5.15}
	\begin{aligned}
		g(z(t),\bar{z}(t)) &= \overline{\rho^{*}}^{T}(0) \cdot f_{0}(z(t),\bar{z}(t)) \\
		&= \left(1,\overline{\rho_{2}^{*}},\overline{\rho_{3}^{*}},\overline{\rho_{4}^{*}}\right)\bar{G}
		\begin{pmatrix}
			-\iota\phi_3(0)e^{-m\phi_3(-\tau)}\phi_1(0)\\
			-\sigma_0\iota\phi_3(0)e^{-m\phi_3(-\tau)}\phi_2(0) \\
			\iota\phi_3(0)e^{-m\phi_3(-\tau)}\big(\phi_1(0)+\sigma_0\phi_2(0)\big)\\
			-\dfrac{\gamma_0  \alpha}{(1+\alpha I^*)^3} \phi_3^2(0)
		\end{pmatrix},
	\end{aligned}
\end{equation}
in which
$$
\begin{aligned}
	\phi_{1}(0)&=z+\bar{z}+W_{20}^{(1)}(0)\frac{z^{2}}{2}+W_{11}^{(1)}(0)z\bar{z}+W_{02}^{(1)}(0)\frac{\bar{z}^{2}}{2}+\cdots,\\
	\phi_{2}(0)&=\rho_{2}z+\bar{\rho}_{2}\bar{z}+W_{20}^{(2)}(0)\frac{z^{2}}{2}+W_{11}^{(2)}(0)z\bar{z}+W_{02}^{(2)}(0)\frac{\bar{z}^{2}}{2}+\cdots,\\
	\phi_{3}(0)&=\rho_{3}z+\bar{\rho}_{3}\bar{z}+W_{20}^{(3)}(0)\frac{z^{2}}{2}+W_{11}^{(3)}(0)z\bar{z}+W_{02}^{(3)}(0)\frac{\bar{z}^{2}}{2}+\cdots,\\
	\phi_{3}(-\tau)&=\rho_{3}z+\bar{\rho}_{3}\bar{z}+W_{20}^{(3)}(-\tau)\frac{z^{2}}{2}+W_{11}^{(3)}(-\tau)z\bar{z}+W_{02}^{(3)}(-\tau)\frac{\bar{z}^{2}}{2}+\cdots,
\end{aligned}$$

Based on formulas \eqref{5.14} and \eqref{5.15}, we arrive at
$$
\begin{aligned}
	g(z(t),\overline{z}(t))=&\overline{G}\biggr\{\left(\overline{\rho_3^*}-1\right)\iota e^{-m\phi_3(-\tau)}\phi_1(0)\phi_3(0)+\left(\overline{\rho_3^*}-\overline{\rho_2^*}\right)\sigma_0\iota e^{-m\phi_3(-\tau)}\phi_2(0)\phi_3(0)\\
	&-\frac{\gamma_0 \alpha}{(1+\alpha I^*)^3}\phi_3^2 (0)\overline{\rho_4^*}\biggr\}\\
	=&\overline{G}\biggr\{\biggr(\left(\overline{\rho_3^*}-1\right)\rho_3\iota+\left(\overline{\rho_3^*}-\overline{\rho_2^*}\right)\sigma_0\iota\rho_2\rho_3-\dfrac{\gamma_0  \alpha\overline{\rho_4^*}\rho_3^2}{(1+\alpha I^*)^3}\biggr)z^{2}\\
	&+\biggr(2\iota\overline{\rho}_3\left(\overline{\rho_3^*}-1\right)+2\left(\overline{\rho_3^*}-\overline{\rho_2^*}\right)\sigma_0\iota\rho_2\overline{\rho}_3-\dfrac{2\gamma_0  \alpha\overline{\rho_4^*}\rho_3\overline{\rho}_3}{(1+\alpha I^*)^3}\biggr)z\overline{z}\\
\end{aligned}
$$	
$$
\begin{aligned}
	&+\biggr(\left(\overline{\rho_3^*}-1\right)\iota\overline{\rho}_3+\left(\overline{\rho_3^*}-\overline{\rho_2^*}\right)\sigma_0\iota\overline{\rho}_2\overline{\rho}_3-\dfrac{\gamma_0  \alpha\overline{\rho_4^*}{\overline{\rho}_3^2}}{(1+\alpha I^*)^3}\biggr){\overline{z}^2}\\
	&+\left(\overline{\rho_3^*}-1\right)\iota \left(W_{11}^{(3)}(0)+\frac{1}{2}W_{20}^{(3)}(0)+\frac{1}{2}\overline{\rho}_3 W_{20}^{(1)}(0)+\rho_3 W_{11}^{(1)}(0)-3m\rho_3\overline{\rho}_3\right)z^2\overline{z}\\
	&+\left(\overline{\rho_3^*}-\overline{\rho_2^*}\right)\sigma_0\iota \left(\rho_2 W_{11}^{(3)}(0)+\frac{1}{2}\overline{\rho}_2W_{20}^{(3)}(0)+\frac{1}{2}\overline{\rho}_3 W_{20}^{(2)}(0)+\rho_3 W_{11}^{(2)}(0)-3m\rho_2\rho_3\overline{\rho}_3\right)z^2\overline{z}\\
	&-\left[\dfrac{2\gamma_0  \alpha\overline{\rho_4^*}\rho_3}{(1+\alpha I^*)^3}W_{11}^{(3)}(0)+\dfrac{\gamma_0  \alpha\overline{\rho_4^*}\overline{\rho}_3}{(1+\alpha I^*)^3}W_{20}^{(3)}(0)\right]z^2\overline{z}+\cdots\biggr\},
\end{aligned}
$$
where the relevant coefficients are expressed as
$$\begin{aligned}
	g_{20}=&2\overline{G}\biggr(\left(\overline{\rho_3^*}-1\right)\rho_3\iota+\left(\overline{\rho_3^*}-\overline{\rho_2^*}\right)\sigma_0\iota\rho_2\rho_3-\dfrac{\gamma_0  \alpha\overline{\rho_4^*}\rho_3^2}{(1+\alpha I^*)^3}\biggr),\\
	g_{11}=&\overline{G}\biggr(\left(\overline{\rho_3^*}-1\right)\iota\overline{\rho}_3+\left(\overline{\rho_3^*}-\overline{\rho_2^*}\right)\sigma_0\iota\overline{\rho}_2\overline{\rho}_3-\dfrac{\gamma_0  \alpha\overline{\rho_4^*}{\overline{\rho}_3^2}}{(1+\alpha I^*)^3}\biggr),\\
	g_{02}=&2\overline{G}\biggr(\left(\overline{\rho_3^*}-1\right)\iota\overline{\rho}_3+\left(\overline{\rho_3^*}-\overline{\rho_2^*}\right)\sigma_0\iota\overline{\rho}_2\overline{\rho}_3-\dfrac{\gamma_0  \alpha\overline{\rho_4^*}{\overline{\rho}_3^2}}{(1+\alpha I^*)^3}\biggr),\\
	g_{21}=&2\overline{G}\biggr(\left(\overline{\rho_3^*}-1\right)\iota \left(W_{11}^{(3)}(0)+\frac{1}{2}W_{20}^{(3)}(0)+\frac{1}{2}\overline{\rho}_3 W_{20}^{(1)}(0)+\rho_3 W_{11}^{(1)}(0)-3m\rho_3\overline{\rho}_3\right)\\
	&+\left(\overline{\rho_3^*}-\overline{\rho_2^*}\right)\sigma_0\iota \left(\rho_2 W_{11}^{(3)}(0)+\frac{1}{2}\overline{\rho}_2W_{20}^{(3)}(0)+\frac{1}{2}\overline{\rho}_3 W_{20}^{(2)}(0)+\rho_3 W_{11}^{(2)}(0)-3m\rho_2\rho_3\overline{\rho}_3\right)\\
	&-\left[\dfrac{2\gamma_0  \alpha\overline{\rho_4^*}\rho_3}{(1+\alpha I^*)^3}W_{11}^{(3)}(0)+\dfrac{\gamma_0  \alpha\overline{\rho_4^*}\overline{\rho}_3}{(1+\alpha I^*)^3}W_{20}^{(3)}(0)\right]\biggr).
\end{aligned}
$$

The computation of $g_{21}$ requires the explicit expressions of $W_{20}(q)$ and $W_{11}(q)$. Combining the mathematical relations \eqref{5.10} and \eqref{5.12}, one derives
\begin{equation}\label{5.16}
	\dot{W} = \dot{u}_t - \dot{z}\rho -\dot{\bar{z}} \bar{\rho} =
	\begin{cases}
		\mathcal{A}(0) W - 2 \operatorname{Re} \big\{\overline{\rho^{*}}(0) \cdot F_{0}(z, \overline{z})\rho(q)\big\}, & q\in [-\tau,0), \\
		\mathcal{A}(0) W -2 \operatorname{Re} \big\{\overline{\rho^{*}}(0) \cdot F_{0}(z, \overline{z})\rho(q)\big\} + f_0, &q= 0.
	\end{cases}
\end{equation}

From formula \eqref{5.13}, we have
\begin{equation}\label{5.17}
	\begin{aligned}
		\dot{W}= \partial_{z}W\dot{z}+\partial_{\bar{z}}W\dot{\overline{z}}=&\big(W_{20}(q)z+W_{11}(q)\overline{z}+\cdots\big)\big(iw_{0}z(t)+g(z,\overline{z})\big) \\
		&+\big(W_{11}(q)z+W_{02}(q)\overline{z}+\cdots\big)\big(-iw_{0}\overline{z}(t)+\overline{g}(z,\overline{z})\big). 
\end{aligned}\end{equation}

By substituting formulas \eqref{5.13} and \eqref{5.17} into equation \eqref{5.16} and matching the corresponding coefficients of $z^2$ and $z \overline{z}$, we can obtain
\begin{equation}\label{5.18}
	(2i \omega_0 I - \mathcal{A}(0)) W_{20}(q) =
	\begin{cases}
		- g_{20} \rho(q) - \overline{g}_{02} \overline{\rho}(q), & q\in [-\tau,0), \\
		- g_{20} \rho(q) - \overline{g}_{02} \overline{\rho}(q) + f_{20}, & q = 0,
	\end{cases}
\end{equation}
and
\begin{equation}\label{5.19}
	- \mathcal{A}(0) W_{11}(q) =
	\begin{cases}
		- g_{11} \rho(q) - \overline{g}_{11}\overline{\rho}(q), & q\in [-\tau,0), \\
		- g_{11} \rho(q) - \overline{g}_{11} \overline{\rho}(q) + f_{11}, &q= 0.
	\end{cases}
\end{equation}

According to the definition of operator $A(0)$ for $q\in [-\tau,0)$, together with \eqref{5.18} and \eqref{5.19}, we get
$$\dot{W}_{20} = 2 i w_0 W_{20}(q) + g_{20} \rho(q) + \overline{g}_{02} \overline{\rho}(q),
$$
and
$$\dot{W}_{11} = g_{11} \rho(q) + \overline{g}_{11} \overline{\rho}(q).
$$

Solving the above differential equations yields
\begin{equation}\label{5.20}
	W_{20}(q)=\frac{ig_{20}\rho(0)}{w_{0}}e^{iw_{0}q}+\frac{i\overline{g}_{02}\overline{\rho}(0)}{3w_{0}}e^{-iw_{0}q}+H_{1}e^{2iw_{0}q},
\end{equation}
\begin{equation}\label{5.21}
	W_{11}(q)=-\frac{ig_{11}\rho(0)}{w_{0}}e^{iw_{0}q}+\frac{i\overline{g}_{11}\overline{\rho}(0)}{w_{0}}e^{-iw_{0}q}+H_{2},
\end{equation}
where $H_{i}=(H_{i}^{(1)},H_{i}^{(2)},H_{i}^{(3)},H_{i}^{(4)})^{T}\in \mathbb{R}^{4}\ (i=1,2)$ denote constant vectors.

We next determine the specific expressions of $H_{1}$ and $H_{2}$. Based on the definition of $A(0)$ at $q=0$ and relation \eqref{5.18}, we establish
\begin{equation}\label{5.22}
	\int_{-\tau_0}^{0}\mathrm{d}\eta(q)W_{20}(q)=2iw_{0}W_{20}(0)+g_{20}\rho(0)+\overline{g}_{02}\overline{\rho}(0)-f_{20}.
\end{equation}

Substituting \eqref{5.20} into \eqref{5.22} and utilizing the identity $iw_{0}I-\int_{-\tau_0}^{0}e^{iw_{0}q}\mathrm{d}\eta(q)\rho(0)=0$, we derive
$$\left(2iw_{0}I-\int_{-\tau_0}^{0}e^{2iw_{0}q}\mathrm{d}\eta(q)\right)H_{1}=2\begin{pmatrix}-\iota e^{-mI^*} \rho_3 \\
	-\sigma_0\iota e^{-mI^*} \rho_2 \rho_3 \\
	\iota e^{-mI^*} \rho_3 (1 + \sigma_0\rho_2) \\
	-\dfrac{\gamma_0 \alpha}{(1+\alpha I^*)^3} \rho_3^2
\end{pmatrix},$$
which further gives
$$
H_1 = 2{D_1}^{-1}
\begin{pmatrix}
	-\iota e^{-mI^*} \rho_3 \\
	-\sigma_0\iota e^{-mI^*} \rho_2 \rho_3 \\
	\iota e^{-mI^*} \rho_3 (1 + \sigma_0\rho_2) \\
	-\dfrac{\gamma_0 \alpha}{(1+\alpha I^*)^3} \rho_3^2
\end{pmatrix}.
$$

The structure of matrix $D_1$ is given by
$$D_1=\begin{pmatrix}
	2iw_0 + D_{111} & -\upsilon & D_{113} & -\vartheta \\
	-\varpi & 2iw_0 + D_{122} & D_{123} & 0 \\
	\iota I^*e^{-mI^*}e^{-2iw_0\tau_0} & -\sigma_0\iota I^*e^{-mI^*}e^{-2iw_0\tau_0}  & 2iw_0+D_{133} & 0 \\
	0 & 0 & -\frac{\gamma_0 }{(1+\alpha I^*)^2} & 2iw_0 +( d_0+\vartheta)
\end{pmatrix},$$
where $D_{111}=( d_0+\varpi)+\iota I^*e^{-mI^*}e^{-2iw_0\tau_0}$, $D_{113}=\left(\iota S^* e^{-mI^*}-\iota I^* S^* m e^{-mI^*}\right)e^{-2iw_0\tau_0}$, $D_{122}=( d_0+\upsilon)+\sigma_0\iota I^*e^{-mI^*}e^{-2iw_0\tau_0}$, $D_{123}=\left(\sigma_0\iota V^* e^{-mI^*}-\sigma_0\iota I^* V^* m e^{-mI^*}\right)e^{-2iw_0\tau_0}$, $D_{133}= ( d_0+\gamma_0 +d)-\iota(S^*+\sigma_0 V^*) e^{-mI^*}e^{-2iw_0\tau_0}+m\iota(S^*+\sigma_0 V^*) e^{-mI^*}I^*e^{-2iw_0\tau_0}$.

Similarly, combining the definition of $A(0)$ at $q = 0$ with formula \eqref{5.19}, we obtain
\begin{equation}\label{5.23}
	\int_{-\tau_0}^{0} \mathrm{d}\eta(q) W_{11}(q) = g_{11}\rho(0) + \overline{g}_{11} \overline{\rho}(0) - f_{11}.
\end{equation}

Substituting \eqref{5.21} into \eqref{5.23} and using $-iw_0I - \int_{-\tau_0}^{0} e^{-iw_0q}\mathrm{d}\eta(q) \overline{\rho}(0) = 0$, we have
$$
\int_{-\tau_0}^{0} \mathrm{d}\eta(q) H_2 = -\begin{pmatrix}
	-\iota e^{-mI^*} (\rho_3 + \bar{\rho}_3) \\
	-\sigma_0\iota e^{-mI^*} (\rho_3\bar{\rho}_2 + \bar{\rho}_3\rho_2) \\
	\iota e^{-mI^*} \big[(\rho_3 + \bar{\rho}_3) + \sigma_0(\rho_3\bar{\rho}_2 + \bar{\rho}_3\rho_2)\big] \\
	-\dfrac{2\gamma_0 \alpha |\rho_3|^2}{(1+\alpha I^*)^3}
\end{pmatrix},$$
which leads to
$$
H_2=-
D_2^{-1}
\begin{pmatrix}
	-\iota e^{-mI^*} (\rho_3 + \bar{\rho}_3) \\
	-\sigma_0\iota e^{-mI^*} (\rho_3\bar{\rho}_2 + \bar{\rho}_3\rho_2) \\
	\iota e^{-mI^*} \big[(\rho_3 + \bar{\rho}_3) + \sigma_0(\rho_3\bar{\rho}_2 + \bar{\rho}_3\rho_2)\big] \\
	-\dfrac{2\gamma_0 \alpha |\rho_3|^2}{(1+\alpha I^*)^3}
\end{pmatrix},
$$
where
$$
D_2=\begin{pmatrix}
	( d_0+\varpi)+\iota I^*e^{-mI^*}& -\upsilon & \left(\iota S^* e^{-mI^*}-\iota I^* S^* m e^{-mI^*}\right) & -\vartheta \\
	-\varpi & ( d_0+\upsilon)+\sigma_0\iota I^*e^{-mI^*} & \left(\sigma_0\iota V^* e^{-mI^*}-\sigma_0\iota I^* V^* m e^{-mI^*}\right)& 0 \\
	\iota I^*e^{-mI^*} &-\sigma_0\iota I^*e^{-mI^*}&D_{233}& 0 \\
	0 & 0 & -\frac{\gamma_0 }{(1+\alpha I^*)^2} & ( d_0+\vartheta)
\end{pmatrix},$$
and $D_{233}= ( d_0+\gamma_0 +d)-\iota(S^*+\sigma_0 V^*) e^{-mI^*}+m\iota(S^*+\sigma_0 V^*) e^{-mI^*}I^*.$

Finally, we derive the explicit expressions of key bifurcation coefficients $\Gamma_1$, $\Gamma _2$ and $T_1$ as follows:
$$\begin{aligned}	& C_{1}(0)=\frac{i}{2w_{0}}\left(g_{20}g_{11}-2|g_{11}|^{2}-\frac{|g_{02}|^{2}}{3}\right)+\frac{g_{21}}{2},\qquad\Gamma _1=-\frac{\operatorname{Re}(C_{1}(0))}{\operatorname{Re}(\lambda^{\prime}(\tau_{0}))}, \\
	& \Gamma_2=2\operatorname{Re}(C_{1}(0)),\qquad T_1=-\frac{\operatorname{Im}\{C_{1}(0)\}+\Gamma_1\operatorname{Im}\left\{\lambda^{\prime}(\tau_{0})\right\}}{w_{0}}.
\end{aligned}$$

\begin{thm}
	The following statements are valid for system \eqref{1.1}.
	
	\begin{enumerate}
		\item[(i)]
		If $\Gamma_1>0$ (respectively, $\Gamma_1<0$), then the Hopf bifurcation is supercritical (respectively, subcritical), and a family of periodic solutions bifurcates from $E^*$ when $\tau>\tau_0$.
		
		\item[(ii)]
		If $\Gamma_2<0$ (respectively, $\Gamma _2>0$), the bifurcating periodic solutions are asymptotically stable (respectively, unstable).
		
		\item[(iii)]
		If $T_1<0$ (respectively, $T_1>0$), the period of bifurcating periodic solutions decreases (respectively, increases).
	\end{enumerate}
\end{thm}
\section{Global Hopf Bifurcation}
Theorem 5.1 reveals that a family of periodic solutions emanates from the equilibrium \(E^*\) as the time delay \(\tau\) crosses the local Hopf bifurcation values $ \tau_n^{(k)}, n=0,1,2,\cdots$. Based on Wu's global Hopf bifurcation theory \cite[Theorem 3.3]{wu1998symmetric}, this study aims to investigate the global dynamical properties of the periodic solutions arising from the bifurcation.

In this section, we consistently suppose that the condition $\mathcal{R}_0>1$ holds true. Then we follow the analytical methods developed in \cite{WANG20231,Threshold,jie,Song2018,zhang2022global}. Let $z(t)=(S(\tau t),V(\tau t), I(\tau t), R(\tau t))^{T}$ denote the rescaled state vector. After performing the time-scaling transformation, the original system \eqref{1.1} is rewritten as
\begin{equation}\label{6.1}
	\frac{d z(t)}{d t}=F\left(z_{t}, \tau, T\right), \quad(t, \tau, T) \in \mathbb{R}_{+} \times(0, \infty) \times \mathbb{R}_{+},
\end{equation}
where the function space is defined as $X:=C([-1,0], \mathbb{R}_{+}^{4})$. For any $\theta \in[-1,0]$, the history segment $z_t\in X$ is defined by $z_{t}(\theta)=z(t+\theta)$, and $T$ represents the period of nonconstant periodic solutions of system \eqref{6.1}. The nonlinear function $F$ takes the explicit form
\begin{equation}\label{6.2}
	F(z_{t}, \tau, T)=\tau \left(\begin{array}{c}
		\Lambda - \iota z_{3t}(0)e^{-mz_{3 t}(-1)}z_{1 t}(0) - ( d_0 + \varpi)z_{1 t}(0) + \upsilon z_{2 t}(0) + \vartheta z_{4 t}(0) \\
		\varpi z_{1 t}(0) - \sigma_0 \iota z_{3 t}(0)e^{-mz_{3 t}(-1)}z_{2 t}(0) - ( d_0 + \upsilon)z_{2 t}(0) \\
		\iota z_{3 t}(0)e^{-mz_{3 t}(-1)}\big(z_{1 t}(0) + \sigma_0 z_{2 t}(0)\big) - ( d_0 + \gamma_0  + d)z_{3 t}(0)\\
		\frac{\gamma_0  z_{3 t}(0)}{1 + \alpha z_{3 t}(0)} - ( d_0 + \vartheta)z_{4 t}(0)
	\end{array}\right),
\end{equation}
with $z_{t}=(z_{1 t}, z_{2 t}, z_{3 t}, z_{4 t}) \in X$. By restricting the domain of $F$ to the finite-dimensional space $\mathbb{R}^{4}$, we obtain the reduced mapping
\[
\tilde{F}(z,\tau ,T):=F|_{\mathbb{R}^{4} \times (0,\infty ) \times \mathbb{R}_{+}}=\tau \left( \begin{array}{c}
	\Lambda - \iota z_{3 }e^{-mz_{3 }}z_{1} - ( d_0 + \varpi)z_{1} + \upsilon z_{2 } + \vartheta z_{4 } \\
	\varpi z_{1 }- \sigma_0 \iota z_{3 }e^{-mz_{3}}z_{2 } - ( d_0 + \upsilon)z_{2} \\
	\iota z_{3}e^{-mz_{3 }}\big(z_{1} + \sigma_0 z_{2 }\big) - ( d_0 + \gamma_0  + d)z_{3 } \\
	\frac{\gamma_0  z_{3 }}{1 + \alpha z_{3}} - ( d_0 + \vartheta)z_{4}
\end{array} \right).
\]

It is straightforward to verify that $\tilde{F}$ is a $C^2$-smooth function fulfilling assumption (A1) in \cite{wu1998symmetric}.

Based on Theorems 4.1 and 4.2, the collection of all equilibrium solutions to system \eqref{6.1} can be described as
\[
\mathcal{N}(F) = \big\{ (E^0, \tau, T),\ (E^*, \tau, T) \,\big|\, (\tau, T) \in (0, \infty) \times \mathbb{R}_+ \big\}.
\]

For each given equilibrium solution $(\tilde{z}, \tau, T) \in \mathcal{N}(F)$, the corresponding characteristic matrix is formulated as
\[
\begin{aligned}
	\Delta_{(\tilde{z},\tau,T)}(\lambda) &= \lambda \mathrm{I} - DF(\tilde{z}, \tau, T)(e^{\lambda \cdot} \mathrm{I})\\
	&=\begin{pmatrix}
		\lambda+\tau\tilde{A}\tilde{z}_3  +\tau( d_0+\varpi) & -\upsilon\tau & \tau\tilde{A}\tilde{z}_1 -m\tau\tilde{A}\tilde{z}_1 \tilde{z}_3 e^{-\lambda} & -\tau\vartheta \\
		-\varpi\tau &\lambda +\tau\sigma_0\tilde{A}\tilde{z}_3+\tau( d_0+\upsilon) & \tau\sigma_0 \tilde{A}\tilde{z}_2 -m\tau \sigma_0 \tilde{A}\tilde{z}_2\tilde{z}_3 e^{-\lambda} & 0 \\
		-\tau\tilde{A}\tilde{z}_3 & -\tau\sigma_0\tilde{A}\tilde{z}_3 &\lambda+ m\tau\tilde{C}\tilde{z}_3  e^{-\lambda} & 0 \\
		0 & 0 & -\tau\dfrac{\gamma_0 }{(1+\alpha \tilde{z}_3 )^2} & \lambda+\tau( d_0+\vartheta)
	\end{pmatrix},
\end{aligned}
\]
where $\mathrm{I}$ stands for the $4 \times 4$ identity matrix, and the simplified coefficients are defined as
$\tilde{A} = \iota e^{-m\tilde{z}_3}$,
$\tilde{C} = d_0+\gamma_0 +d$.
Accordingly, the characteristic equation corresponding to the steady-state solution $(\tilde{z}, \tau, T)$ is given by
\[
\det \Delta_{(\tilde{z},\tau,T)}(\lambda) = \lambda^4 + \tilde{a}_{11}\tau\lambda^3 + \tilde{a}_{12}\tau^2\lambda^2 + \tilde{a}_{13}\tau^3 \lambda+\tilde{a}_{14}\tau^4+ e^{-\lambda}(\tilde{a}_{21}\tau\lambda^3 + \tilde{a}_{22}\tau^2\lambda^2 + \tilde{a}_{23}\tau^3\lambda+\tilde{a}_{24}\tau^4) = 0,
\]
where the coefficients $\tilde{a}_{1i}$ and $\tilde{a}_{2i}$ ($i=1,2,3,4$) are given by:
\begin{align*}
	\tilde{a}_{11}
	&=\tilde{A}\tilde{z}_3+d_0+\varpi+\sigma_0\tilde{A}\tilde{z}_3+d_0+\upsilon+d_0+\vartheta,\\
	\tilde{a}_{12}
	&=\bigl(\tilde{A}\tilde{z}_3+d_0+\varpi\bigr)\bigl(\sigma_0\tilde{A}\tilde{z}_3+d_0+\upsilon\bigr)
	+\sigma_0^2\tilde{A}^2\tilde{z}_2\tilde{z}_3
	-\upsilon\varpi
	+\tilde{A}^2\tilde{z}_1\tilde{z}_3\\
	&\quad+\bigl(d_0+\vartheta\bigr)\bigl(\tilde{A}\tilde{z}_3+d_0+\varpi+\sigma_0\tilde{A}\tilde{z}_3+d_0+\upsilon\bigr),\\
	\tilde{a}_{13}
	&=\bigl(d_0+\vartheta\bigr)\Big\{
	\bigl(\tilde{A}\tilde{z}_3+d_0+\varpi\bigr)\bigl(\sigma_0\tilde{A}\tilde{z}_3+d_0+\upsilon\bigr)
	+\sigma_0^2\tilde{A}^2\tilde{z}_2\tilde{z}_3
	-\upsilon\varpi
	+\tilde{A}^2\tilde{z}_1\tilde{z}_3
	\Big\}\\
	&\quad+\bigl(\tilde{A}\tilde{z}_3+d_0+\varpi\bigr)\sigma_0^2\tilde{A}^2\tilde{z}_2\tilde{z}_3
	+\upsilon\sigma_0\tilde{A}^2\tilde{z}_2\tilde{z}_3
	+\sigma_0\varpi\tilde{A}^2\tilde{z}_1\tilde{z}_3
	+\tilde{A}^2\tilde{z}_1\tilde{z}_3\bigl(\sigma_0\tilde{A}\tilde{z}_3+d_0+\upsilon\bigr)\\
	&\quad-\frac{\gamma_0}{(1+\alpha\tilde{z}_3)^2}\vartheta\tilde{A}\tilde{z}_3,\\
	\tilde{a}_{14}
	&=\bigl(d_0+\vartheta\bigr)\Big\{
	\sigma_0^2\tilde{A}^2\tilde{z}_2\tilde{z}_3\bigl(\tilde{A}\tilde{z}_3+d_0+\varpi\bigr)
	+\upsilon\sigma_0\tilde{A}^2\tilde{z}_2\tilde{z}_3
	+\sigma_0\varpi\tilde{A}^2\tilde{z}_1\tilde{z}_3
	+\tilde{A}^2\tilde{z}_1\tilde{z}_3\bigl(\sigma_0\tilde{A}\tilde{z}_3+d_0+\upsilon\bigr)
	\Big\}\\
	&\quad-\frac{\gamma_0}{(1+\alpha\tilde{z}_3)^2}\vartheta\Big\{
	\sigma_0\varpi\tilde{A}\tilde{z}_3+\tilde{A}\tilde{z}_3\bigl(\sigma_0\tilde{A}\tilde{z}_3+d_0+\upsilon\bigr)
	\Big\},\\
	\tilde{a}_{21}
	&=m\tilde{C}\tilde{z}_3,\\
	\tilde{a}_{22}
	&=m\tilde{C}\tilde{z}_3\bigl(\tilde{A}\tilde{z}_3+d_0+\varpi+\sigma_0\tilde{A}\tilde{z}_3+d_0+\upsilon+d_0+\vartheta\bigr)
	-m\sigma_0^2\tilde{A}^2\tilde{z}_2\tilde{z}_3^2
	-m\tilde{A}^2\tilde{z}_1\tilde{z}_3^2,\\
	\tilde{a}_{23}
	&=m\tilde{C}\tilde{z}_3(d_0+\vartheta)\bigl(\tilde{A}\tilde{z}_3+d_0+\varpi+\sigma_0\tilde{A}\tilde{z}_3+d_0+\upsilon\bigr)
	-(d_0+\vartheta)\bigl(m\sigma_0^2\tilde{A}^2\tilde{z}_2\tilde{z}_3^2+m\tilde{A}^2\tilde{z}_1\tilde{z}_3^2\bigr)\\
	&\quad
	+m\tilde{C}\tilde{z}_3\bigl(\tilde{A}\tilde{z}_3+d_0+\varpi\bigr)\bigl(\sigma_0\tilde{A}\tilde{z}_3+d_0+\upsilon\bigr)
	-m\sigma_0^2\tilde{A}^2\tilde{z}_2\tilde{z}_3^2\bigl(\tilde{A}\tilde{z}_3+d_0+\varpi\bigr)
	-m\tilde{C}\tilde{z}_3\upsilon\varpi\\
	&\quad-m\sigma_0\tilde{A}^2\tilde{z}_2\tilde{z}_3^2\upsilon
	-m\sigma_0\varpi\tilde{A}^2\tilde{z}_1\tilde{z}_3^2
	-m\tilde{A}^2\tilde{z}_1\tilde{z}_3^2\bigl(\sigma_0\tilde{A}\tilde{z}_3+d_0+\upsilon\bigr)
	,\\
	\tilde{a}_{24}
	&=(d_0+\vartheta)\Big\{
	m\tilde{C}\tilde{z}_3\bigl(\tilde{A}\tilde{z}_3+d_0+\varpi\bigr)\bigl(\sigma_0\tilde{A}\tilde{z}_3+d_0+\upsilon\bigr)
	-m\sigma_0^2\tilde{A}^2\tilde{z}_2\tilde{z}_3^2\bigl(\tilde{A}\tilde{z}_3+d_0+\varpi\bigr)
	-m\tilde{C}\tilde{z}_3\upsilon\varpi\\
	&\quad-m\sigma_0\tilde{A}^2\tilde{z}_2\tilde{z}_3^2\upsilon
	-m\sigma_0\varpi\tilde{A}^2\tilde{z}_1\tilde{z}_3^2
	-m\tilde{A}^2\tilde{z}_1\tilde{z}_3^2\bigl(\sigma_0\tilde{A}\tilde{z}_3+d_0+\upsilon\bigr)
	\Big\}.
\end{align*}

When $\mathcal{R}_0>1$, $m \geq \frac{\sigma_0 \iota \varpi\left(1-\sigma_0\right)}{\left(d_0+\upsilon+\varpi\right)\left(d_0+\upsilon+\sigma_0 \varpi\right)}$ and condition (H)  holds, zero is never an eigenvalue for all equilibrium points associated with system \eqref{6.1}. This consequently validates hypothesis (A2) from \cite{wu1998symmetric}. Moreover, the smoothness condition (A3) required in \cite{wu1998symmetric} is directly validated by the expression of $F$ in \eqref{6.2}.

Following the definitions in \cite{wu1998symmetric}, a steady-state solution $(\tilde{z}, \tilde{\tau}, \tilde{T})$ of \eqref{6.1} is called a center if
\[
\det \Delta_{(\tilde{z},\tilde{\tau},\tilde{T})}\left(i \frac{2j\pi}{\tilde{T}}\right) = 0
\]
for some $j\in\mathbb{N}$. Moreover, the center is called isolated if there exists an open neighborhood around $(\tilde{z}, \tilde{\tau}, \tilde{T})$ that contains no other centers and has only finitely many purely imaginary characteristic roots of the form $i\frac{2j\pi}{\tilde{T}}$. We denote by $J(\tilde{z}, \tilde{\tau}, \tilde{T})$ the collection of all positive integers $j$ satisfying the above condition.

From Theorem 5.1, for each integer $n \geq 0$, the point $\left(E^*, \tau_n^{(k)}, \frac{2\pi}{w_k \tau_n^{(k)}}\right)$ is an isolated center of system \eqref{6.1}. Moreover, there exists exactly one pair of purely imaginary roots in the form $i j \frac{2\pi}{\tilde{T}}$, corresponding to $j=1$ and $\tilde{T} = \frac{2\pi}{w_k \tau_n^{(k)}}$. Thus, we have
\begin{equation}\label{6.3}
	J(\tilde{z}, \tilde{\tau}, \tilde{T}) = \{1\}.
\end{equation}

In addition, Theorem 5.1 ensures that the crossing number corresponding to every isolated center satisfies
\begin{equation}\label{6.4}
	\gamma_1\left(E^*, \tau_n^{(k)}, \frac{2\pi}{w_k\tau_n^{(k)}}\right) = -1.
\end{equation}
This result verifies that assumption (A4) in \cite{wu1998symmetric} is fulfilled.

We next construct $\Sigma(F)$, a closed subset embedded within the product space $X\times(0,\infty)\times\mathbb{R}_+$, via the topological closure operation as specified below:
\[
\Sigma(F) = \mathrm{Cl}\left\{(z_t, \tau, T) \in X \times (0, \infty) \times \mathbb{R}_+ : z_t \text{ is a nontrivial periodic solution of \eqref{6.1} with period } T\right\}.
\]

For any nonnegative integer $n \ge 0$, we use $\mathcal{C}\left(E^*, \tau_n^{(k)}, \frac{2\pi}{w_k\tau_n^{(k)}}\right)$ to stand for the connected branch inside $\Sigma(F)$ that passes through the pivotal bifurcation point $\left(E^*, \tau_n^{(k)}, \frac{2\pi}{w_k\tau_n^{(k)}}\right)$. The assertions derived in Theorem 5.1 guarantee this connected set cannot be empty. By applying the global periodic bifurcation principle established in reference \cite{wu1998symmetric}, only two mutually exclusive outcomes are mathematically admissible:
\begin{enumerate}[(i)]
	\item $\mathcal{C}\left(E^*, \tau_n^{(k)}, \frac{2\pi}{w_k\tau_n^{(k)}}\right)$ is unbounded in $X\times(0,\infty)\times\mathbb{R}_+$;
	
	\item $\mathcal{C}\left(E^*, \tau_n^{(k)}, \frac{2\pi}{w_k\tau_n^{(k)}}\right)$ is bounded, and
	\[
	\sum_{(\tilde z,\tau,T)\in \mathcal{C}\left(E^*,\tau_n^{(k)},\frac{2\pi}{w_k\tau_n^{(k)}}\right)\cap \mathcal{N}(F)} \gamma_j(\tilde z,\tau,T)=0,
	\]where $\gamma_j(\tilde z,\tau,T)$ denotes the $j$-th crossing number at the equilibrium.
\end{enumerate}

From Theorem 5.1, we know that on $\mathcal{C}\left(E^*,\tau_n^{(k)},\frac{2\pi}{w_k\tau_n^{(k)}}\right)$, the only relevant crossing number is $\gamma_1=-1$. Therefore,
\[
\sum_{(\tilde z,\tau,T)\in \mathcal{C}\left(E^*,\tau_n^{(k)},\frac{2\pi}{w_k\tau_n^{(k)}}\right)\cap \mathcal{N}(F)} \gamma_j(\tilde z,\tau,T)=-1\neq0,
\]
which implies that case (ii) cannot occur. Consequently, case (i) must hold. The following two lemmas are used to verify that the projections of $\mathcal{C}\left(E^*, \tau_n^{(k)},\frac{2\pi}{w_k\tau_n^{(k)}}\right)$ onto the $\tau$-space and $T$-space are bounded, respectively.

\begin{lem}
	All nontrivial periodic solutions of system \eqref{6.1} are uniformly bounded in $C([-1,0],\mathbb{R}^4)$.
\end{lem}

\begin{proof}
	By Theorem 2.2, every solution satisfies
	\[
	S(t)+V(t)+I(t)\leq \frac{\Lambda}{d_0},\quad t\geq 0,
	\]
	which gives a uniform upper bound for each state component.
	Suppose a nontrivial periodic solution satisfies $I(t)\equiv 0$ for all $t$. Substituting $I\equiv 0$ into the system yields a linear autonomous subsystem for $S,V,R$ whose unique equilibrium is the disease-free steady state $E^0$. This contradicts the definition of nontrivial periodic solutions, so any nontrivial periodic orbit cannot vanish identically in the infected component.
	It follows from Theorem 4.3 that the system is uniformly persistent along all such periodic trajectories, which provides a positive uniform lower bound for all state variables.
	Consequently, all periodic solutions are contained in a compact subset of $\mathbb{R}_+^4$, hence they are uniformly bounded in $C([-1,0],\mathbb{R}^4)$. This completes the proof.
\end{proof}

\begin{lem}
	Suppose that $\mathcal{R}_0>1$,
	\[
	m \geq \frac{\sigma_0 \iota \varpi\left(1-\sigma_0\right)}{\left(d_0+\upsilon+\varpi\right)\left(d_0+\upsilon+\sigma_0 \varpi\right)},
	\]
	and conditions \eqref{eq:H} and \eqref{eq:P} are satisfied. Then system \eqref{6.1} has no nonconstant periodic solution with period equal to $1$.
\end{lem}

\begin{proof}
	Assume, conversely, that $z(t) = (z_1(t), z_2(t), z_3(t), z_4(t))$ is a nonconstant periodic solution of period $1$ to \eqref{6.1}. Then $z(t-1)=z(t)$ for all $t$. Substituting this identity into \eqref{6.1}, the functional differential equation reduces to the ordinary differential system
	\begin{equation}\label{6.5}
		\begin{cases}
			\dfrac{dz_1(t)}{dt} = \tau\Lambda  - \tau\iota z_{3}(t)e^{-mz_{3}(t)}z_{1}(t) - \tau( d_0 + \varpi)z_{1}(t) + \tau\upsilon z_{2}(t) + \tau\vartheta z_{4}(t),\\[6pt]
			\dfrac{dz_2(t)}{dt} = \tau \varpi z_{1}(t)- \tau\sigma_0 \iota z_{3}(t)e^{-mz_{3}(t)}z_{2}(t) - \tau( d_0 + \upsilon)z_{2}(t),\\[6pt]
			\dfrac{dz_3(t)}{dt} =\tau\iota z_{3}(t)e^{-mz_{3}(t)}\big(z_{1}(t) + \sigma_0 z_{2}(t)\big) - \tau( d_0 + \gamma_0  + d)z_{3}(t),\\[6pt]
			\dfrac{dz_4(t)}{dt} =\tau\dfrac{\gamma_0 z_{3}(t)}{1 + \alpha z_{3}(t)} - \tau( d_0 + \vartheta)z_{4}(t).
		\end{cases}
	\end{equation}
	
	By Theorem 4.4, system \eqref{6.5} possesses a unique globally asymptotically stable positive equilibrium. A globally asymptotically stable equilibrium precludes the existence of nonconstant periodic solutions, which contradicts our initial assumption. This completes the proof.
\end{proof}

Lemma 6.1 ensures that the projection of $\mathcal{C}\left(E^*, \tau_n^{(k)}, \frac{2\pi}{w_k\tau_n^{(k)}}\right)$ onto the state space $C([-1,0],\mathbb{R}^4)$ is bounded for all $n\ge 0$.
By Lemma 6.2, there exist no nonconstant $1$-periodic solutions for system \eqref{6.1}.
If a nonconstant solution possessed period $\frac{1}{n+1}$, then $1$ would also be a period of that solution, which is prohibited by Lemma~6.2.
Hence, nonconstant periodic solutions with period $\frac{1}{n+1}$ cannot exist.
Combining inequality \eqref{5.7}, we obtain
\[
\frac{1}{n+1}<\frac{2\pi}{w_k\tau_n^{(k)}}<1,\quad n=1,2,\dots,
\]
which implies that the projection onto the $T$-space is also bounded. Consequently, the projection onto the $\tau$-space must be unbounded. This completes the proof of the global bifurcation result.

\begin{thm}
		Suppose that $\mathcal{R}_0>1$,
	\[
	m \geq \frac{\sigma_0 \iota \varpi\left(1-\sigma_0\right)}{\left(d_0+\upsilon+\varpi\right)\left(d_0+\upsilon+\sigma_0 \varpi\right)},
	\]
	and conditions \eqref{eq:H} and \eqref{eq:P} are satisfied.
	Then the connected component
	\[
	\mathcal{C}\left(E^*, \tau_n^{(k)}, \frac{2\pi}{w_k\tau_n^{(k)}}\right)
	\]
	is unbounded in $C([-1,0],\mathbb{R}^4)\times\mathbb{R}_+\times\mathbb{R}_+$.
	Furthermore, for each integer $n\ge 0$, there exist arbitrarily large $\tau>\tau_n^{(k)}$ such that system \eqref{1.1}  possesses at least one nontrivial periodic solution.
\end{thm}

\section{Numerical Simulations}
The following section presents numerical experiments that illustrate and support the theoretical results obtained above. We conduct the global Hopf bifurcation analysis with the MATLAB package \emph{DDE-BIFTOOL}, which was originally introduced by Engelborghs et al.~\cite{Engelborghs,10.1145/513001.513002}.

The parameters of system \eqref{1.1} used in numerical simulations are given below:
\[
\begin{aligned}
	d_0&=0.08,\ \iota=0.003,\ \varpi=0.8,\ \upsilon=0.01,\ \vartheta=0.4,\\
	\sigma_0&=0.9,\ m=0.1,\ \gamma_0 =0.15,\ \alpha=0.04,\ d=0.05,
\end{aligned}
\]
with initial conditions $S_0 = 40$, $V_0 = 20$, $I_0 = 8$, $R_0 = 4$.

\begin{figure}[htbp] 
	\centering
	\begin{subfigure}{0.3\textwidth}
		\centering
		\includegraphics[width=8cm,height=5cm]{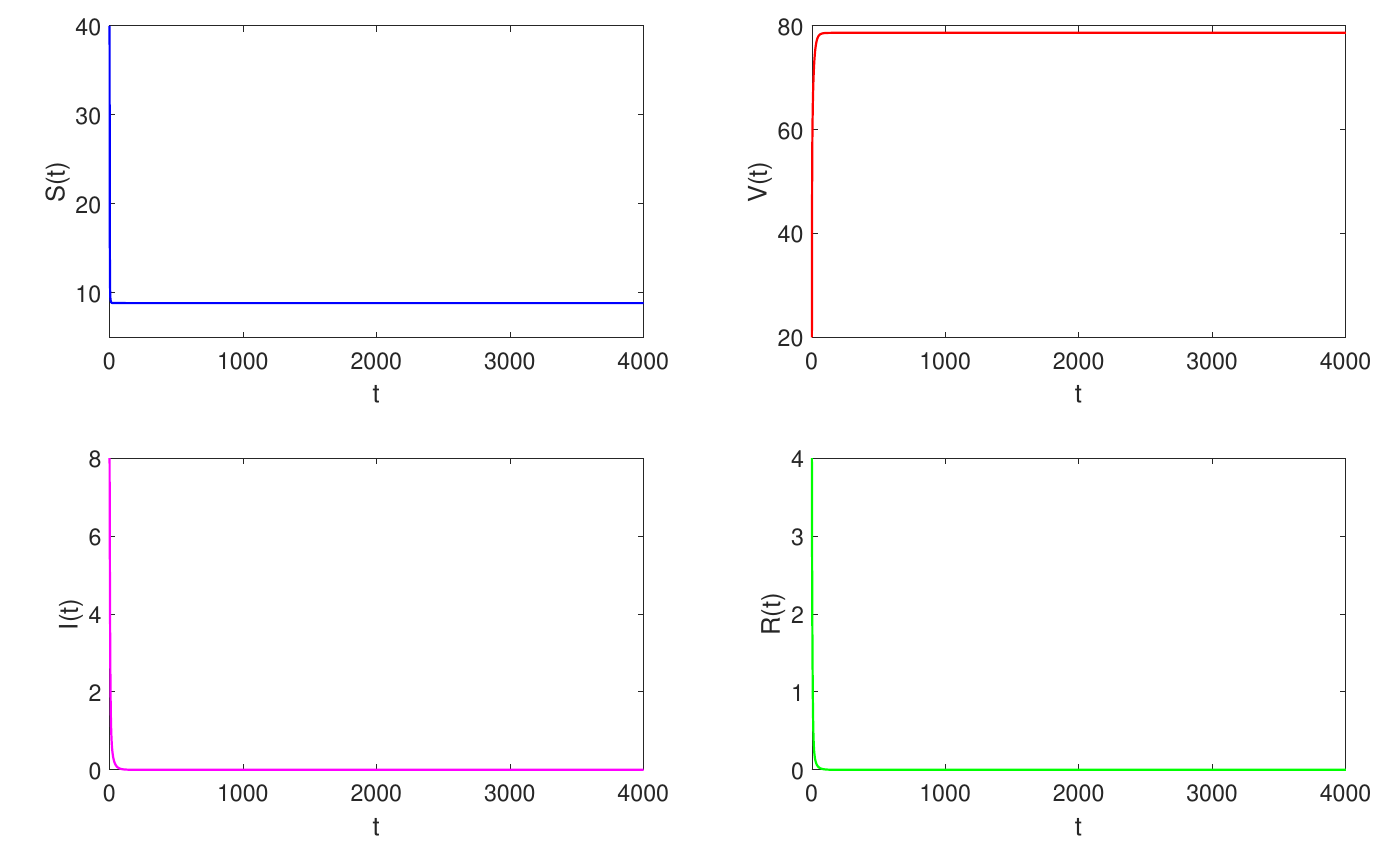}
		\caption{}
		\label{fig:image1}
	\end{subfigure}%
	\hfill
	\begin{subfigure}{0.3\textwidth}
		\centering
		\includegraphics[width=4.5cm,height=4.3cm]{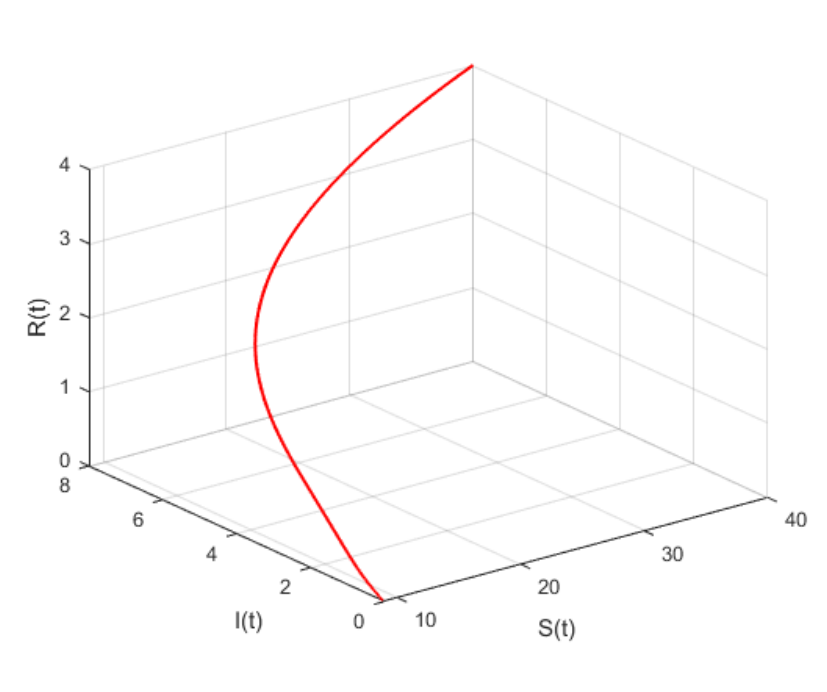}
		\caption{}
		\label{fig:image2}
	\end{subfigure}%
	\vspace*{0.5cm}
	\caption{ For $\Lambda =7$ and $\tau=6$, the disease-free equilibrium $E^0$ is globally asymptotically stable whenever $\mathcal{R}_0 < 1$.}
	\label{fig}
\end{figure}

\begin{figure}[htbp] 
	\centering
	\begin{subfigure}{0.3\textwidth}
		\centering
		\includegraphics[width=8cm,height=5cm]{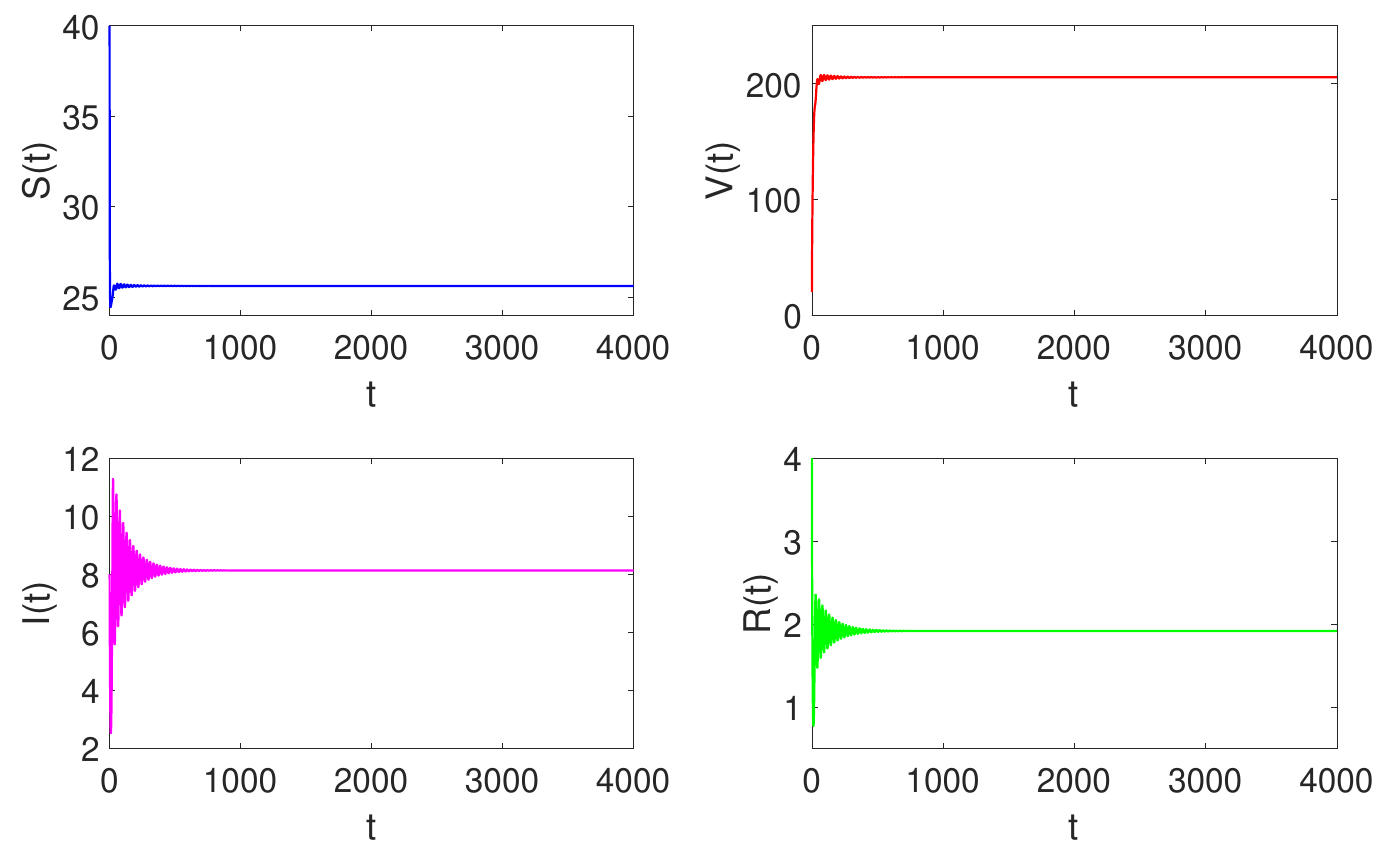}
		\caption{}
		\label{fig:image1}
	\end{subfigure}%
	\hfill
	\begin{subfigure}{0.3\textwidth}
		\centering
		\includegraphics[width=4.5cm,height=4.3cm]{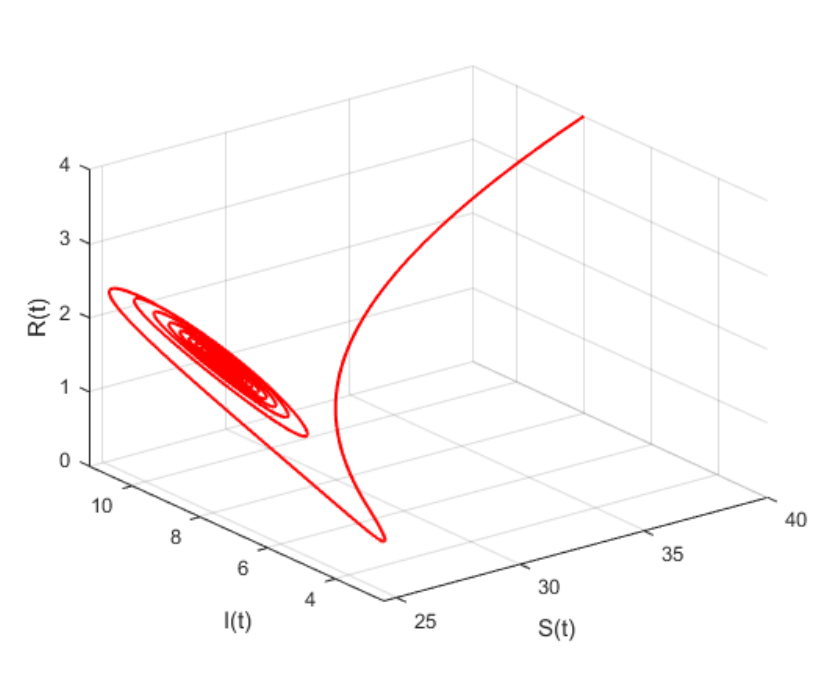}
		\caption{}
		\label{fig:image2}
	\end{subfigure}%
	\vspace*{0.5cm}
	\caption{ For $\Lambda =20$ and $\tau=6.5$, the endemic equilibrium $E^*$ remains globally asymptotically stable when $\mathcal{R}_0 > 1$.}
	\label{fig}
\end{figure}
\begin{figure}[htbp] 
	\centering
	\begin{subfigure}{0.3\textwidth}
		\centering
		\includegraphics[width=8cm,height=5cm]{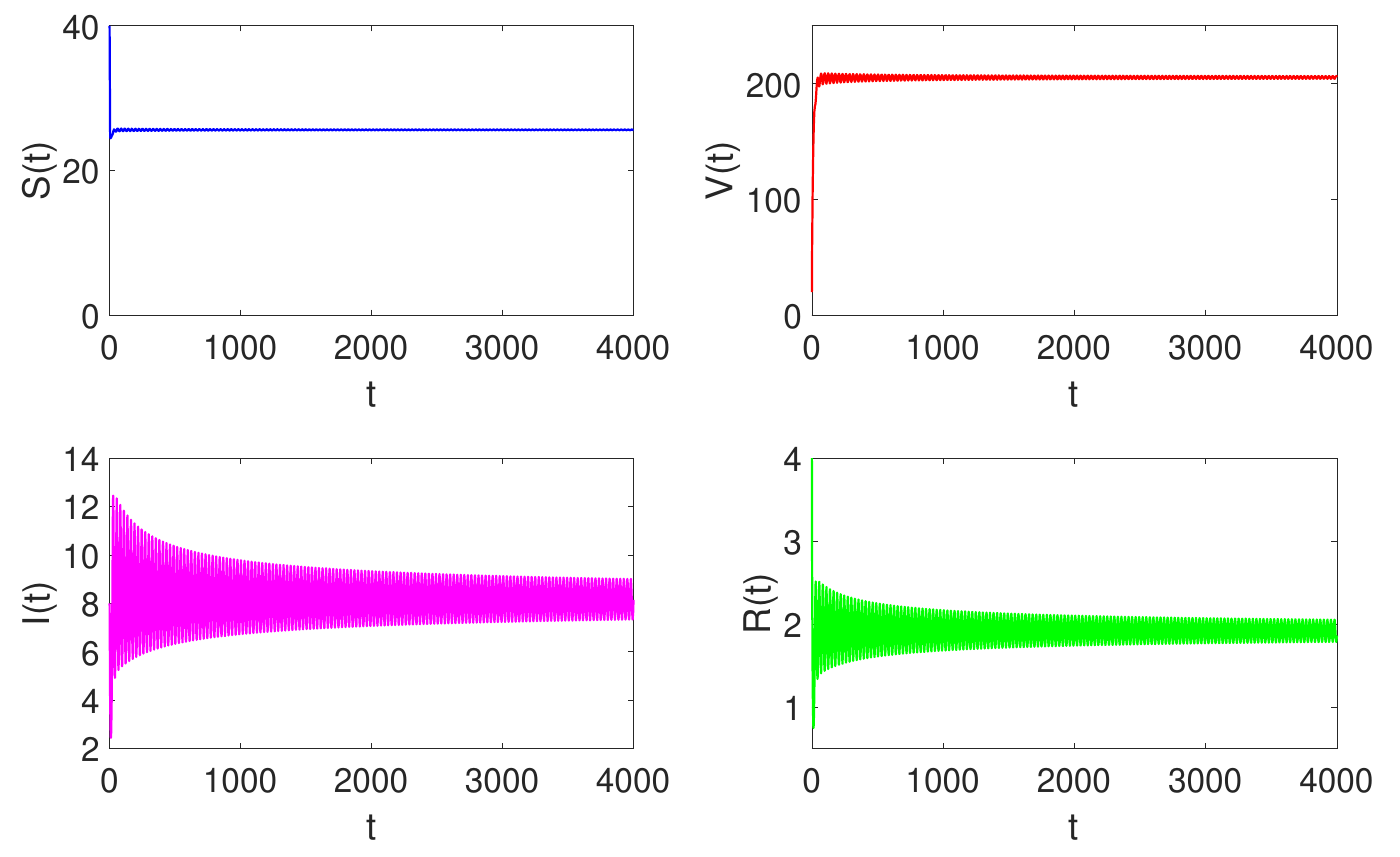}
		\caption{}
		\label{fig:image1}
	\end{subfigure}%
	\hfill
	\begin{subfigure}{0.3\textwidth}
		\centering
		\includegraphics[width=4.5cm,height=4.3cm]{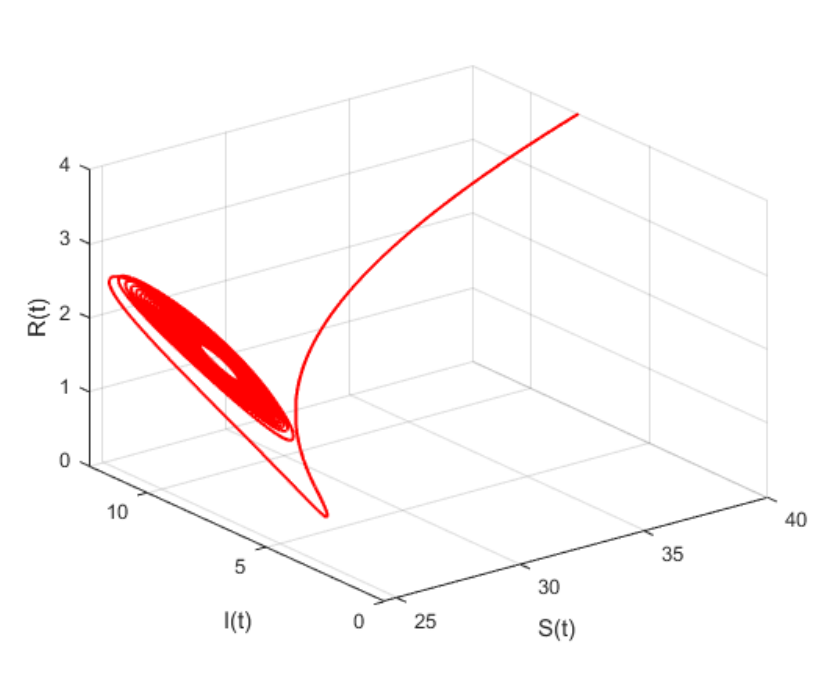}
		\caption{}
		\label{fig:image2}
	\end{subfigure}%
	\vspace*{0.5cm}
	\caption{ For $\Lambda =20$ and $\tau=7$, the endemic equilibrium $E^*$ becomes unstable when $\mathcal{R}_0 > 1$.}
	\label{fig}
\end{figure}
\begin{figure}[htbp] 
	\centering
	\begin{subfigure}{0.3\textwidth}
		\centering
		\includegraphics[width=8cm,height=5cm]{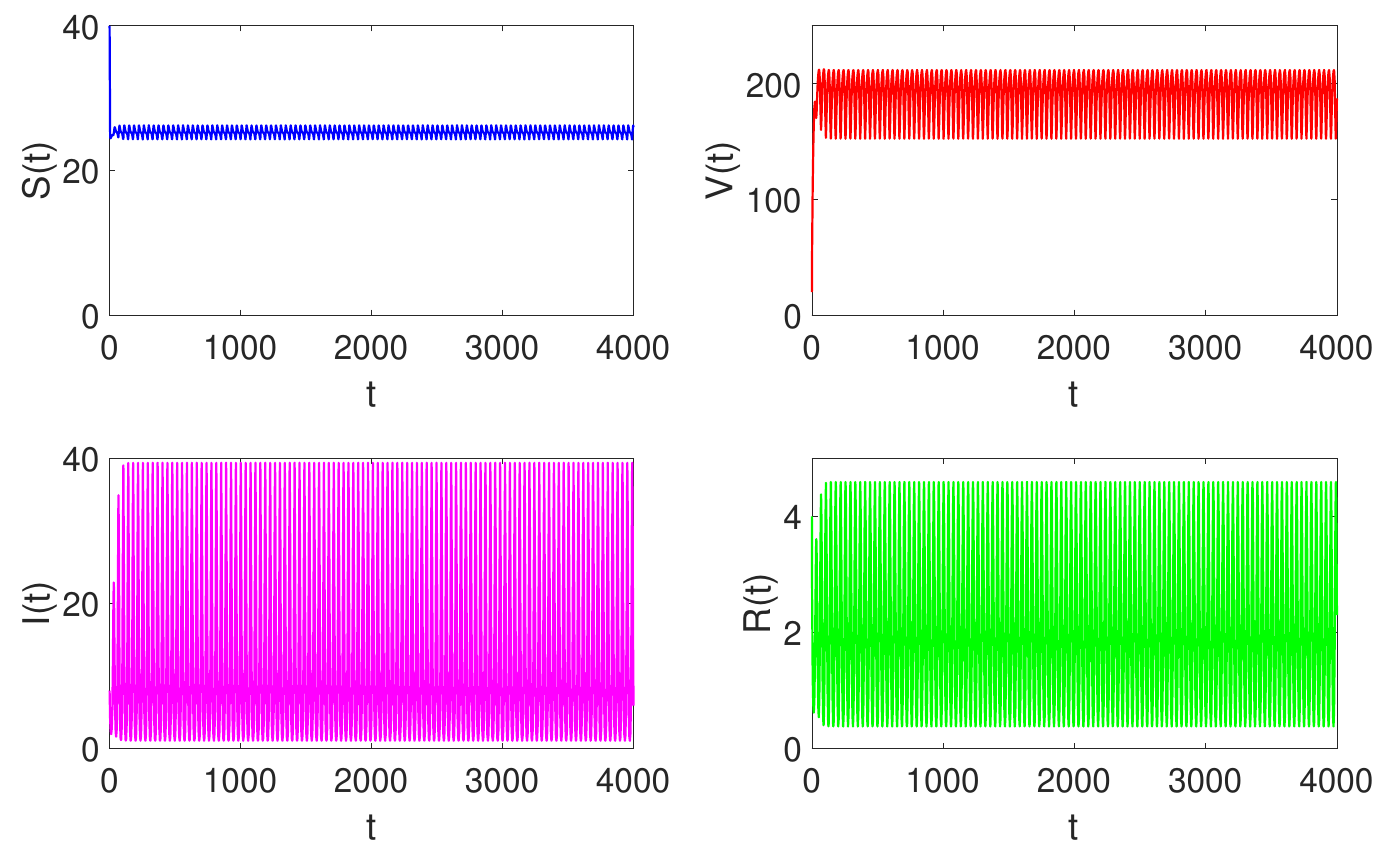}
		\caption{}
		\label{fig:image1}
	\end{subfigure}%
	\hfill
	\begin{subfigure}{0.3\textwidth}
		\centering
		\includegraphics[width=4.5cm,height=4.3cm]{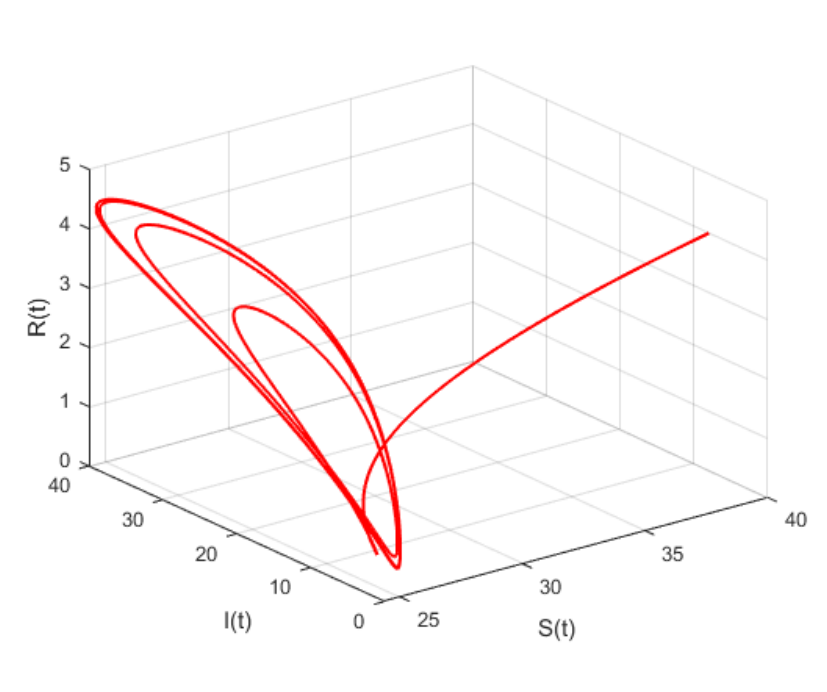}
		\caption{}
		\label{fig:image2}
	\end{subfigure}%
	\vspace*{0.5cm}
	\caption{ For $\Lambda =20$ and $\tau=10$, the endemic equilibrium $E^*$ is unstable when $\mathcal{R}_0 > 1$.}
	\label{fig}
\end{figure}
\begin{figure}
	\centering
	\includegraphics[width=7cm,height=7cm]{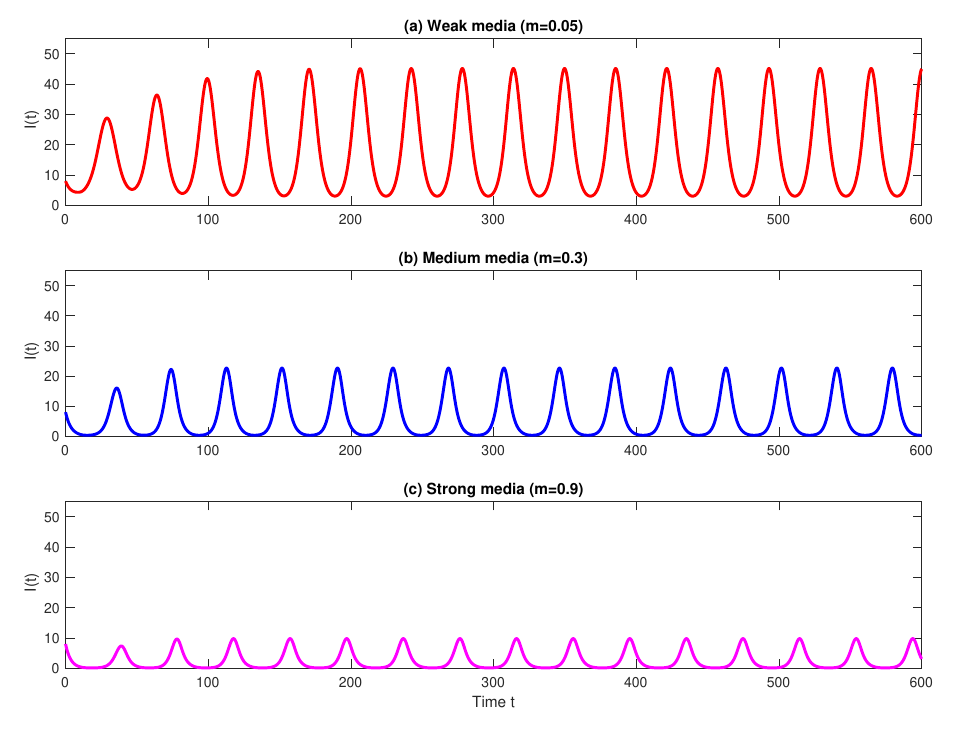}
	\caption{Effect of $m$ on system oscillation amplitude}
	\label{fig:5}
\end{figure}
When $\Lambda =7$ and $\tau=6$, the basic reproduction number is 
$\mathcal{R}_0=0.8532<1$. Moreover,
\[
\frac{\iota\Lambda}{d_0(d_0+\gamma_0+d)}
=
\frac{0.003\times 7}{0.08(0.08+0.15+0.05)}
=0.9375<1.
\]
Therefore, according to Theorem~4.1, the disease-free equilibrium $E^0$ 
is globally asymptotically stable, as illustrated in Figure~2. 
When $\Lambda =20$, we obtain $\mathcal{R}_0=2.4378>1$, which ensures 
the existence and uniqueness of the endemic equilibrium $E^*$. 
By analyzing the characteristic equation, we obtain the critical delay 
$\tau_0=6.9983$, as well as the bifurcation sequence 
$\tau_1=34.0371$, $\tau_2=61.0758,\cdots$. 
Figure~3 shows that the endemic equilibrium 
$E^*=(25.6194,205.4953,8.1362,1.9183)$ is asymptotically stable when 
$\tau=6.5<\tau_0$. When $\tau=7>\tau_0$, $E^*$ loses stability and a 
Hopf bifurcation occurs, accompanied by the appearance of stable periodic 
solutions (Figure~4). When $\tau=10>\tau_0$, the system exhibits 
persistent periodic oscillations (Figure~5).

The media impact coefficient $m$ has a significant regulatory effect on the system dynamics. As shown in Figure 6, larger values of $m$ suppress the amplitude of oscillations and exert a strong stabilizing effect on the epidemic system.

\begin{figure}[htbp]  
	\centering
	\includegraphics[width=0.5\textwidth]{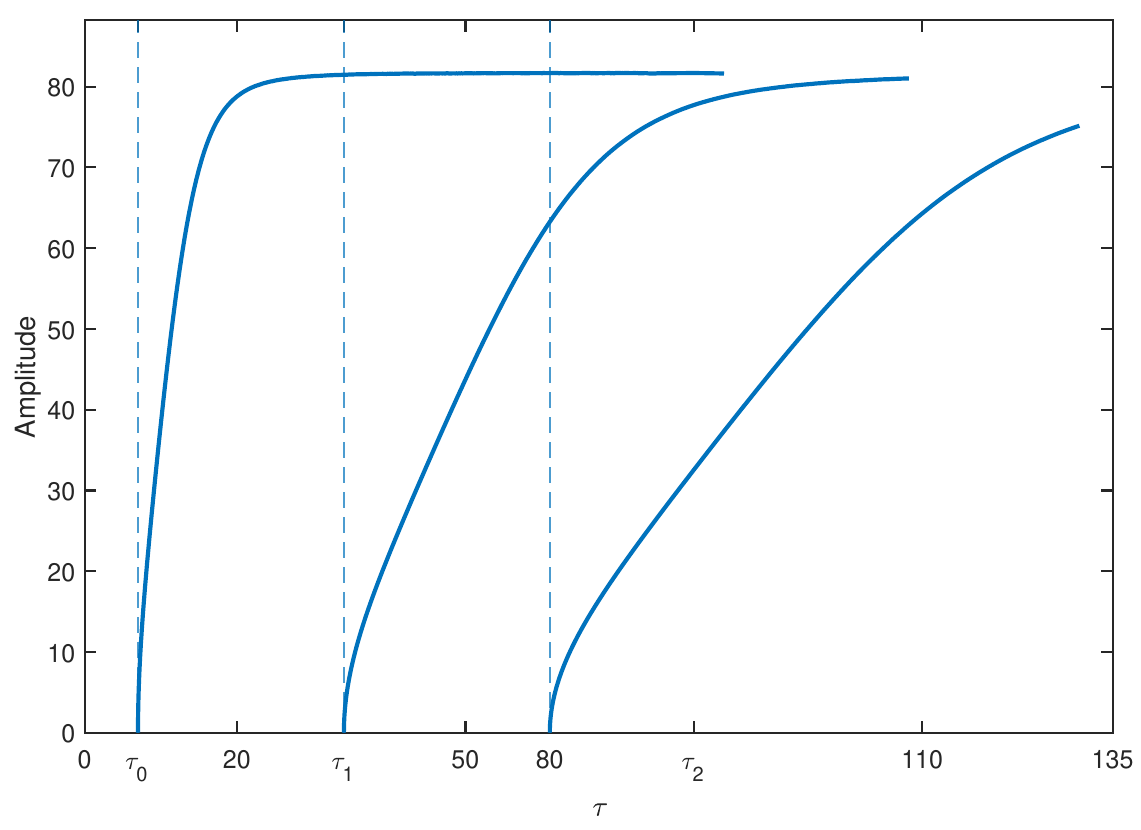} 
	\caption{Periodic solution branches for system \eqref{1.1} corresponding to $\tau_0=6.9983$, $\tau_1=34.0371$, and $\tau_2=61.0758$.}
	\label{fig:6}
\end{figure}
\begin{figure}
	\centering
	\includegraphics[width=0.5\textwidth]{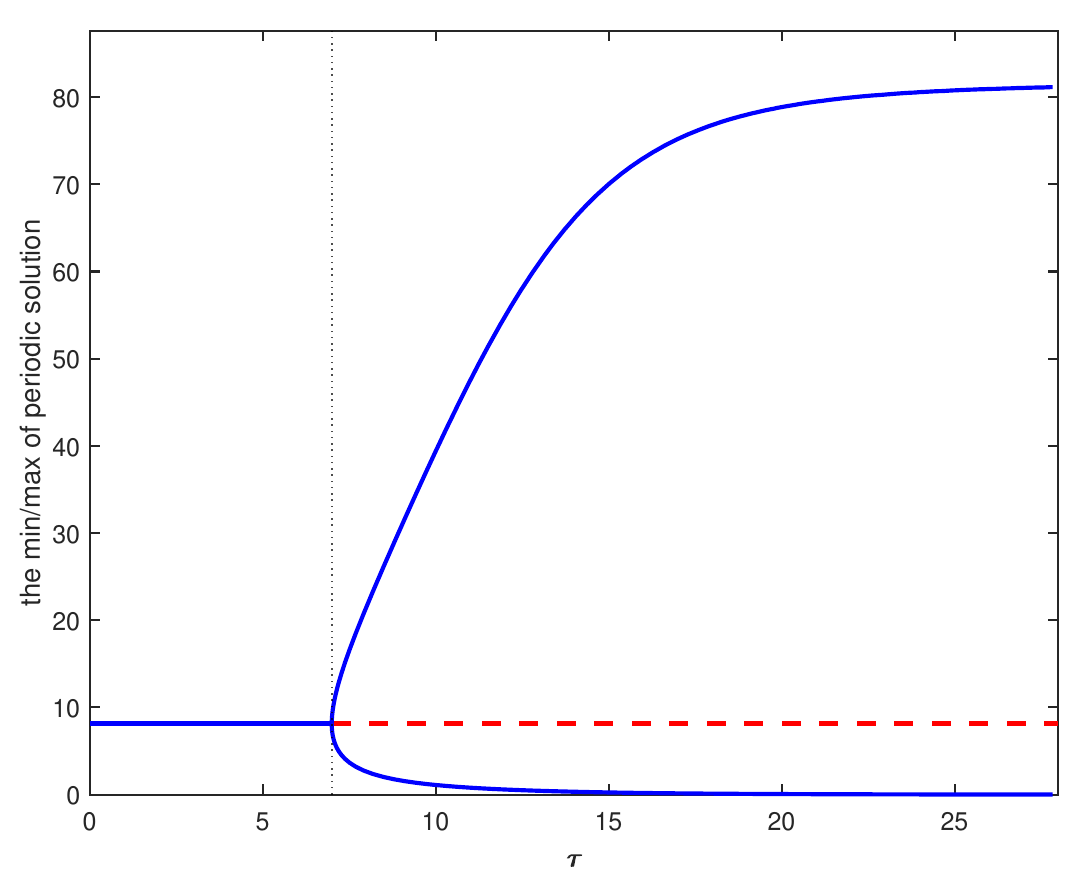}
	\caption{Bifurcation diagram of system \eqref{1.1}; red dashed curves denote unstable equilibria.}
	\label{fig:6}
\end{figure}
To verify the global continuation of bifurcating periodic solutions, we present the global Hopf branches at $\tau_0=6.9983$, $\tau_1=34.0371$, and $\tau_2=61.0758$ in Figure 7. These branches extend continuously and remain unbounded as the time delay $\tau$ increases, which indicates that the periodic solutions generated by local Hopf bifurcation do not terminate in a finite interval of $\tau$, but exist globally for sufficiently large delays. This numerically confirms the conclusion of Theorem 6.1 that the connected component of periodic solutions is unbounded.

Finally, by selecting $\tau$ as the bifurcation parameter, we present the corresponding one-parameter bifurcation diagram in Figure 8, in which the red dashed curve denotes the unstable endemic equilibrium. It clearly illustrates the stability transition of $E^*$, the emergence of Hopf bifurcation, and the global continuation of periodic oscillations as $\tau$ increases.
\section{Conclusion}
This paper focuses on the bifurcation dynamics of an SVIRS epidemic system with delayed media coverage. The model incorporates vaccination, temporary immunity, media-related behavioral responses, and the saturation effect of medical treatment, providing a more realistic description of epidemic transmission.

We rigorously establish the positivity and uniform boundedness of solutions and derive the basic reproduction number $\mathcal{R}_0$. The disease-free equilibrium is shown to be globally asymptotically stable under a sufficient condition stronger than $\mathcal{R}_0<1$. For $\mathcal{R}_0>1$, the existence and uniqueness of an endemic equilibrium and uniform persistence of the disease are established under suitable additional conditions. Moreover, under further sufficient conditions, the endemic equilibrium is globally asymptotically stable.

To explore the influence of delayed media responses, the delay parameter $\tau$ is chosen as the bifurcation parameter. The analysis shows that the endemic equilibrium remains locally stable for small delays, while larger delays may induce periodic oscillations through Hopf bifurcation. The direction and stability of bifurcating periodic solutions are investigated using center manifold reduction and normal form theory. In addition, the global continuation of periodic branches is obtained by applying the global Hopf bifurcation theorem.

Numerical simulations further illustrate the theoretical results and indicate that stronger media coverage can effectively weaken oscillatory outbreaks and enhance the stability of the epidemic system. These findings demonstrate the joint influence of vaccination, media coverage, and delayed behavioral responses on epidemic dynamics and control.
\vskip 20 pt
\noindent{\bf Acknowledgement}
\vskip 10 pt
The first author is partially supported by the National Key Research and Development Program of China (Grant No. 2020YFA0713100).
\nocite{*}
\bibliography{k} 
\end{document}